\documentclass[leqno]{article}

\usepackage{graphicx}
\usepackage{amscd}
\usepackage{amsmath}
\usepackage{caption}
\usepackage{amsfonts}
\usepackage{amssymb}
\usepackage{mathrsfs}
\usepackage{multicol}
\usepackage{color,comment}
\usepackage{pgfplots}
\pgfplotsset{compat=1.18}
\usepackage{upgreek,enumitem,hyperref}
\usepackage{chngcntr}
\counterwithin{figure}{section}

\usepackage{authblk} 
\usepackage{csquotes}
\usepackage[T1]{fontenc}
\usepackage{latexsym}
\usepackage{amssymb,amsthm,amsfonts}
\usepackage{mathtools}
\usepackage{relsize}
\usepackage{framed, color}
\usepackage{dsfont}
\usepackage{todonotes}
\usepackage{cancel}
\usepackage[normalem]{ulem}
\usepackage{xcolor}

\usepackage{gauss} 
\usepackage{geometry}
\usepackage{stmaryrd}
\usepackage{nicefrac}

\numberwithin{equation}{section}

\definecolor{darkgreen}{rgb}{0.09, 0.45, 0.27}
\definecolor{debianred}{rgb}{0.84, 0.04, 0.33}
\definecolor{orange}{rgb}{1.0, 0.5, 0.0}
\newcommand{\D}{\mathcal{D}}
\newcommand{\cC}{\mathcal{C}}

\newcommand{\R}{\mathbb R}

\newcommand{\be}{\begin{equation}}
\newcommand{\ben}{\begin{eqnarray*}}
\newcommand{\een}{\end{eqnarray*}}

\newcommand{\dist}{\mathrm{dist}}

\newtheorem{theorem}{Theorem}[section]
\newtheorem{lemma}[theorem]{Lemma}
\newtheorem{remark}[theorem]{Remark}

\newtheorem{definition}[theorem]{Definition}
\allowdisplaybreaks

\allowdisplaybreaks[1]

\usepackage[normalem]{ulem}
\definecolor{DarkBlue}{rgb}{0,0.1,0.7} 
\definecolor{DarkGreen}{rgb}{0,0.5,0.1}

\newcommand\soutD{\bgroup\markoverwith
	{\textcolor{DarkGreen}{\rule[.5ex]{2pt}{1pt}}}\ULon}

\newcommand{\Lrad}[2][2]{L^{#1}_{\mathrm{rad}}\if\relax\expandafter\detokenize{#2}\relax\else(#2)\fi}
\newcommand{\Lanti}[2][2]{L^{#1}_{\mathrm{anti}}\if\relax\expandafter\detokenize{#2}\relax\else(#2)\fi}
\newcommand{\Lradanti}[2][2]{L^{#1}_{\mathrm{rad},\mathrm{anti}}\if\relax\expandafter\detokenize{#2}\relax\else(#2)\fi}
\newcommand{\Hrad}[2][1]{H^{#1}_{\mathrm{rad}}\if\relax\expandafter\detokenize{#2}\relax\else(#2)\fi}
\newcommand{\Hradanti}[2][1]{H^{#1}_{\mathrm{rad},\mathrm{anti}}\if\relax\expandafter\detokenize{#2}\relax\else(#2)\fi}
\newcommand{\norm}[1]{\left\lVert \if\relax\expandafter\detokenize{#1}\relax\impvar\else #1\fi \right\rVert}
\newcommand{\Norm}[1]{\bigl\lVert \if\relax\expandafter\detokenize{#1}\relax\impvar\else #1\fi \bigr\rVert}

\usepackage{tikz}
\usepackage{pgfplots}
\pgfplotsset{compat=1.18}
\usepgfplotslibrary{fillbetween} 
\usepackage{mathtools}
\usepackage{amssymb}

\newcommand{\scm}{\mathrm{sc}^-}
\newcommand{\hscm}{\hat{\mathrm{sc}}^-}
\newcommand{\conv}{\mathrm{conv}}
\newcommand{\argmin}{\mathrm{argmin}}

\newcommand{\Z}{\mathbb{Z}}

\newcommand{\N}{\mathbb{N}}

\newcommand{\Id}{\text{Id}} %

 \def\dd{\, {\rm d}}

\title{Optimal data-driven solutions for a stationary diffusive model of population growth}

\author[1]{Laura Baldelli}

\author[2]{Paolo Malanchini}

\author[1]{Wolfgang Reichel}

\affil[1]{Institute for Analysis, Karlsruhe Institute of Technology (KIT), D-76128 Karlsruhe, Germany.}
\affil[2]{Dipartimento di Matematica e Applicazioni, Universit\`a degli Studi di Milano - Bicocca, via Roberto Cozzi 55, 20125 - Milano, Italy.}
\date{\small\today}

\begin{document}

\maketitle

\begin{abstract}
We study optimal data-driven solutions for the stationary diffusive population growth model $-\Delta u = r u$ in a bounded domain $\Omega\subset\R^N$ with Neumann boundary conditions. Instead of prescribing a functional relation between the position $x$, the net per-capita growth rate $r$ and the population size $u$, we look for a pair $(u,r)\in H^1(\Omega)\times L^\infty(\Omega)$ that fits a given data set in an optimal way measured by a cost functional $I$ and an additional penalty term. We characterize the relaxed cost functional $\scm I$ by showing that its density is given as the partial lower convex envelope with respect to the variable $r$, and prove the existence of optimal data-driven solutions. Furthermore, we establish a consistency result comparing conventional solutions of $-\Delta u = \varrho(x,u)u$ with optimal data-driven solutions where the data set stems from the functional relation $(x,u)\mapsto \varrho(x,u)$. Finally, as data sets evolve, we prove the convergence of optimal solutions via the $\Gamma$-convergence of the associated cost functionals.
\end{abstract}


\section{Introduction and main results}
{
  \renewcommand{\thefootnote}{}
  \footnotetext{{\bf{MSC 2020}}: 35J20, 49J45.}
  \footnotetext{{\bf{Keywords}}: optimal data-driven solutions, population growth, relaxation, Gamma convergence.}
  \footnotetext{{\bf Corresponding author:} Wolfgang Reichel}
}
\setcounter{footnote}{0}


Within less than a decade data-driven aspects have impacted the classical theory of partial differential equations. The availability of large data sets and the possibility to handle them efficiently have opened new scenarios where classical PDEs models are enhanced by data-driven points of view. The particular mathematical framework for optimal data-driven solutions that we have chosen builds upon the recently introduced data-driven approach to solid mechanics \cite{KO, CMO}, where datasets consist of vector-valued strain–flux pairs $(\varepsilon,\sigma)\in \mathbb{R}^{d}\times\mathbb{R}^{d}$. Subsequently, the authors in \cite{LSS23} extended this methodology to the modeling and analysis of viscous fluid mechanics, where datasets consist of matrix-valued strain–stress pairs $(\varepsilon,\sigma)\in \mathbb{R}^{d\times d}\times\mathbb{R}^{d\times d}$. In contrast to the vector and matrix settings described above, in this paper we consider a scalar data-driven framework applied to a population growth model. Our motivation is twofold: on the one hand, the scalar nature of the underlying elliptic equation avoids major technical complications, making the theoretical framework conceptually more intuitive and transparent. We note that scalar data-driven formulations have also been addressed in applied contexts, such as by R\"oger and Schweizer \cite{RS20} in porous media modeling, and in \cite{SM24} for scalar diffusion problems. On the other hand, population dynamics models are of interest in their own right, as we discuss in more detail below. We believe that the current setting where the functional relation $(x,u)\mapsto \varrho(x,u)$ has been given up in favor of optimal pairs $(r,u)$ in the phase space can be generalized to a large variety of other nonlinear elliptic boundary value problems for scalar of vector valued unknowns, e.g., in models for chemical reactions \cite{Epstein_Pojman}, chemotaxis \cite{Keller_Segel}, water-biomass interaction \cite{Klausmeier}, riot dynamics \cite{riots}, activator-inhibitor stimulated pattern formation \cite{Gierer_Meinhardt, Murray_2}, preditor prey models \cite{Dunbar1, Dunbar2}, competing or cooperating species \cite{Cantrell_Cosner} -- just to name a few. 

\medskip

Let us now describe our setting in more detail. We assume that $\Omega\subset \R^N$, $N\ge 2$ is an open bounded Lipschitz domain.  On the differential equations side we consider the problem
\begin{equation} \label{eq:main} 
-\Delta u = r u \mbox{ in } \Omega, \quad \partial_\nu u  =0 \mbox{ on } \partial \Omega
\end{equation}
with $(u,r)$ from the phase space $V\coloneqq H^1(\Omega)\times L^\infty(\Omega)$. The main requirement for pairs $(u,r)$ is expressed by the differential constraint
\begin{equation}\label{eq:constraint}
\cC=\left\{(u, r)\in V, u\ge 0 : -\Delta u= r u \mbox{ in } \Omega,\quad \partial_\nu u=0 \mbox{ on } \partial \Omega\right\}.
\end{equation}
The model \eqref{eq:main} it the stationary version of a diffusive evolution problem $\partial_t u =\Delta u +ru$, where $u(x,t)$ models the population density and $r(x,t)$ stands for the net per-capita growth rate of the population at time $t$ and position $x\in\Omega$. Members of the population are supposed to be mobile and to prefer areas of low population density. Therefore, a diffusive model is adequate and for simplicity we assume the diffusivity to be normalized by $1$.
 
\medskip

Next we give some definitions on the data side. Data points are supposed to belong to the local phase space $Y\coloneqq \Omega\times[0,\infty)\times\R$ and a data set $\D$ is a set of data points
\begin{equation}
\D=\{(x_\beta, u_\beta, r_\beta)\in Y : \beta\in B\}
\end{equation}
for some index set $B$. A data point $(x_\beta, u_\beta, r_\beta)$ stands for a measurement of population density $u_\beta$ and net per-capita reproduction rate $r_\beta$ at the spatial point $x_\beta\in \Omega$. Note that time plays no role since we are in a stationary setting.

\medskip

In order to see how well a pair $(u,r)\in \cC$ fits to the data set $\D$ we define a cost functional $I:V\to [0,\infty]$ 
\begin{align}\label{eq:I_f}
 I(u,r) \coloneqq 
 \begin{cases} 
 \int_{\Omega} \dist( (x,u(x),r(x)), \D) \dd x, & (u,r) \in \cC, \\ 
 \infty, & \text{otherwise},
 \end{cases}
\end{align}
where $\dist: Y\to [0,\infty)$ is the pseudo-distance\footnote{The pseudo-distance does not satisfy the triangle inequality. Nevertheless, it is more convenient to work with the pseudo-distance than with its square root version.}
\begin{equation}\label{distance}
\dist((x,u,r),\D)\coloneqq \inf_{(\tilde x, \tilde u, \tilde r)\in \D} \left( |x-\tilde x|^2+ |u-\tilde u|^2+|r-\tilde r|^2 \right).
\end{equation}
Note that we have defined the cost functional on the entire phase space $V$ (and not only on $\cC$). Ideally, one would like to consider minimizers of $I$. But in general, the functional $I$ is not weak-weak$^\ast$ lower semi-continuous on $V$ since for a general data set the integrand is not convex. Therefore we need to consider the weak-weak$^\ast$ lower semi-continuous relaxation $\scm I: V\to [0,\infty]$ of $I$ defined by (the weak-weak$^\ast$ lower semi-continuity of $\scm I$ is shown in Lemma~\ref{lemma_sc-} of the Appendix):
\begin{equation} \label{def:sc-}
 \scm I(u,r) \coloneqq \inf\left\{ \liminf_{n\to\infty} I(u_n,r_n) : u_n \rightharpoonup u \mbox{ in } H^1(\Omega) \mbox{ and } r_n\stackrel{*}{\rightharpoonup} r \mbox{ in } L^\infty(\Omega)\right\}.
\end{equation}
Beside the relaxation of $I$ we also add a constraint that enforces the norm bound $\|r\|_\infty\leq R$ for some given $R>0$. This will be done via the penalty functional
$$
\iota_R: L^\infty(\Omega)\to \{0,\infty\}, \quad \iota_R(r)\coloneqq \left\{ \begin{array}{ll} 0 & \mbox{ if } \|r\|_\infty\leq R, \vspace{\jot} \\
\infty & \mbox{ if } \|r\|_\infty >R.
\end{array} \right.
$$
On the one hand, it is reasonable to assume that the net per capita growth rate $r$ cannot be arbitrary large. On the other hand, this restriction will add coercivity in the $r$-direction and it will be useful in Section~\ref{sec_consistency}. With these preparations we can now define \textbf{optimal data-driven solutions} of \eqref{eq:main} as follows. 

\begin{definition}\label{def:ods}
 For given $R>0$ and a given data set $\D\subseteq Y$ a minimizer $(u^\ast,r^\ast)\in \mathcal{C}$ of $\scm I+\iota_R$ is called an optimal data-driven solution of \eqref{eq:main}. 
\end{definition}

One of our main results is that we can precisely identify the weak-weak$^\ast$ lower semi-continuous relaxation $\scm I$. 

\begin{definition} \label{def:conv} Let $f:Y \to [0,\infty)$ be a Carath\'{e}odory-function and let $J:V\to [0,\infty]$ given by
$$
J(u,r) \coloneqq \int_\Omega f(x,u(x),r(x))\dd x.
$$
\begin{itemize}
\item[(a)] The partial lower convex envelope $\conv_r f: Y$ of $f$ is defined as 
$$
\conv_r f(x,u,r) \coloneqq \sup\left\{ g(x,u,r): r \mapsto g(\cdot,\cdot,r) \mbox{ is convex and } g \le f\right\}.
$$
\item[(b)] The functional $J^\ast: V\to [0,\infty]$  is defined as 
    \begin{equation}\label{eq:I_f_r}
 J^\ast (u,r) \coloneqq  \int_{\Omega} \conv_r f(x, u(x),r(x)) \dd x.
  \end{equation}
\end{itemize}
\end{definition}

\begin{theorem} \label{thm:main_Istar_scm_I}
    If $I$ is given as in \eqref{eq:I_f} and $\scm I$ is its weak-weak$^\ast$ lower semi-continuous relaxation then $\scm I = I^\ast$ with $I^\ast$ from Definition~\ref{def:conv}.
\end{theorem}

\begin{remark}
Although the variables $u$ and $r$ appear symmetrically in the cost functional they play very asymmetric roles in our analysis. Namely, compactness in the $u$-variable makes it possible that convexity only with respect to the $r$-variable is sufficient to identify the weak-weak$^\star$ lower semi-continuous envelope $\scm I$ as $I^\ast$. 
\end{remark}

The further main results of our paper can be described as follows:
\begin{itemize}
    \item[(a)] An existence result for optimal data-driven solutions is given in Theorem~\ref{thm:ex_odds}.
    \item[(b)] Consistency: suppose that we know the complete data set $\D=\{(x,u,\varrho(x,u)): x\in \Omega, u\in [0,\infty)\}$ from a known reproduction law $(x,u)\mapsto \varrho(x,u)$. For such a data set we can consider optimal data-driven solutions of \eqref{eq:main} and we can also consider weak solutions of $-\Delta U = \varrho(x,U)U$ in $\Omega$, $\partial_\nu U=0$ on $\partial\Omega$. The question is, under which conditions on the reproduction law $\varrho$ the two notions coincide. The answer is given in Theorem~\ref{thm:consistency}.
    \item[(c)] Convergence: here we consider the question of what happens if data sets $\D_n$ converge (in a sense made precise in Definition~\ref{def_conv_data_set}) to a limit data set $\D$ as $n\to \infty$. In Theorem~\ref{thm:gammaconv} we show that the $\Gamma$-limit of the cost functionals $I_n$ is the right limiting object and we identify the $\Gamma$-limit. In particular this implies the convergence of optimal data-driven solutions.
\end{itemize}
Finally, let us give some remarks on models of population growth. Although the mathematical study of population growth dates back to Fibonacci, the foundational modern formulation is the logistic equation introduced by Verhulst \cite{Verhulst}, given by
$$\partial_t u=a(b-u)u \quad \text{in } \Omega\subset\mathbb{R}^N,$$
where $a,b\in \mathbb{R}^+$. Drawing an analogy with physics, Fisher \cite{Fisher} incorporated a spatial diffusion term into Verhulst's equation, leading to the reaction–diffusion partial differential equation
$$\partial_t u=a(b-u)u+\Delta u \quad \text{in } \Omega\subset\mathbb{R}^N.$$
In this setting, spatial heterogeneities, represented by space-dependent positive coefficients $a(x)$ and $b(x)$ in the nonlinear term, play a crucial role in the behavior of positive solutions, as they reflect the fact that population distribution varies across spatial positions, cf. \cite{Du2006,   Henry1981, CC2003, XS23}. Instead of logistic growth models, a number of different models have been considered in  \cite{KPP}, see also \cite{BacaerBook} for an exhaustive study on the subject.  In the present paper, we focus on stationary growth models featuring spatial heterogeneity, coupled with Neumann boundary conditions to model an isolated environment with zero flux. This setting motivates the problem formulated in \eqref{eq:main} where we drop the exact functional relation between net per-capita growth rate $r$ and population density $u$ in favor of optimal proximity of a pair $(u,r)\in\cC$ to a given data set $\D$.
Extending this framework to the time-dependent case is nevertheless of great interest, as population growth inherently varies over time. Although a space-time variational formulation for evolutionary problems was recently studied in \cite{SX}, applying an approach that preserves causality is far from straightforward in data-driven problems. Consequently, we leave the time-dependent problem for a future work.
\medskip

{\it Notations.} We denote by $B_r(x)$ the open $\R^N$-ball with center $x\in\R^N$ and radius $r>0$. The symbols $\overline{B}_r(x)$, $\partial B_r(x)$, $B_r(x)^c$ stand, respectively, for the closure, the boundary, and the exterior of the ball $B_r(x)$. Let $Q \coloneqq (-1/2, 1/2)^N$ be the unit cube and $Q_{ \delta}(x_0)\coloneqq  x_0 + { \delta}Q$ its translated and scaled copy for a given $x_0\in\R^N$ and $\delta>0$. 

Given any $A\subseteq \R^N$, we write $\chi_A$ for the characteristic function of $A$. For any $N$-dimensional Lebesgue measurable set $A$ let $|A|$ be its $N$-dimensional Lebesgue measure. Given a real-valued function $\varphi$ we define $\varphi_+\coloneqq\max\{\varphi,0\}$ (resp., $\varphi_-\coloneqq\max\{-\varphi,0\}$) as its positive (resp., negative) part.

We denote by $H^1(\Omega)$ and $L^q(\Omega)$ with $1\le q\le \infty$ the standard Sobolev and Lebesgue spaces endowed with the norms $\|\cdot\|_{H^1}$ and $\|\cdot\|_{q}$, respectively.
On the phase space $V= H^1(\Omega)\times L^\infty(\Omega)$ we have the weak-weak$^\ast$ topology. Since $H^1(\Omega)$ is a reflexive Sobolev space and $L^\infty(\Omega)$ is the dual of the separable space $L^1(\Omega)$, their respective weak and weak$^\ast$ topologies are metrizable on bounded sets. We denote by $d_w$, $d_{w*}$, $d$ the local metrics in $H^1(\Omega)$, $L^\infty(\Omega)$ and $H^1(\Omega)\times L^\infty(\Omega)$ respectively. 

For the sake of simplicity, we will use the following notation to indicate the pseudo distance $\dist((x,u,r),\D)$ defined in \eqref{distance} with possible frozen variables 
\begin{equation}\label{def:dist:fix}
f(x,u,r)=\dist((x,u,r),\D) \quad f_u(\cdot, \cdot)=\dist((\cdot,u,\cdot),\D),\quad f_{x,u}(\cdot)=\dist((x,u,\cdot),\D).
\end{equation}

\medskip
The structure of the paper is as follows. In Section~\ref{sec_existence}, after establishing some properties of the dataset $\D$ and the constraint set $\cC$, we prove the existence of an optimal data-driven solution of \eqref{eq:main}.
In Section~\ref{sec_consistency}, we deal with consistency, proving Theorem~\ref{thm:consistency}. We address convergence in Section~\ref{sec_conv}, proving the $\Gamma$-convergence of the associated cost functionals to $\scm I$. Then, Section~\ref{sec_scmI=Istar} is devoted to the proof of Theorem~\ref{thm:main_Istar_scm_I}, that is the identification of $\scm I$ with $I^*$. Finally, in the Appendix we prove the weak-weak$^\ast$ lower semi-continuity of $\scm I$ together with its characterization.

\section{Fundamental properties of $\cC, \D, I, \scm I$ and existence}\label{sec_existence}

In this section we show the existence result for optimal data-driven solutions in Theorem~\ref{thm:ex_odds}. It requires a number of fundamental properties such as weak-weak$^\ast$ closedness of $\cC$, coercivity properties of $I, \scm I$ and weak-weak$^\ast$ lower semi-continuity of $\scm I$. We begin with the following result on the coercivity properties of the functionals $I, \scm I$. 

To prove the existence of an optimal data-driven solution to equation \eqref{eq:main}, as defined in Definition \ref{def:ods}, it is actually sufficient to consider the coercivity of the operator $I + \iota_R$ and $\scm I + \iota_R$. However, we also prove the coercivity of $I$ and $\scm I$ under slightly different assumptions which will be needed in Section \ref{sec_conv}.

\begin{lemma}\label{l:coercivity_omega}
If there exists $a\in \R, b>0$ such that 
\begin{itemize}
\item[(a)] $\D$ satisfies 
\begin{equation} \label{eq:condD}
 \D \subseteq \Omega \times \mathcal S_{a,b} \mbox{ with } \mathcal S_{a,b}\coloneqq \{ (u,r)\in [0,\infty)\times \R \,: u^2 < aur+b \}
 \end{equation}
then $\scm I+\iota_R$ is coercive,
\item[(b)] $\D$ satisfies 
\begin{equation} \label{eq:condD_new}
 \D \subseteq \Omega \times \Sigma_{a,b} \mbox{ with } { \Sigma}_{a,b}\coloneqq \{ (u,r)\in [0,\infty)\times \R \,: u^2 +r^2 < aur+b \}
 \end{equation}
then $I$ and $\scm I$ are coercive.
\end{itemize}
\end{lemma} 

\begin{remark}
 The set $\mathcal S_{a,b}$ consists of the points bounded by the $r$-axis and the right branch of a hyperbola. If $a=0$, however, $\mathcal S_{a,b}$ degenerates into the region between two vertical lines. Regarding $\Sigma_{a,b}$, it contains the points within an ellipse if $|a|<2$, and the region between the two branches of a hyperbola otherwise.
See Figures~\ref{fig:data:sets1} and \ref{fig:data:sets2} for a possible shape of $\mathcal S_{a,b}$ and $\Sigma_{a,b}$. 
\end{remark}

\begin{figure}[h!]
 \label{fig:data:sets}
 \centering
 \begin{minipage}{0.48\textwidth}
  \centering
  \begin{tikzpicture}[scale=0.85]
      \begin{axis}[
      axis lines = middle,
      xlabel = {$u$},
      ylabel = {$r$},
      ymin = -4, ymax = 8,
      xmin = -1, xmax = 7,
      xtick = {1,2,3,4,5,6},
      ytick = {-3,-2,-1,1,2,3,4,5,6,7},
      axis equal image,
      inner axis line style={-stealth},
      axis on top,
      ]
      \clip (axis cs:0,-4) rectangle (axis cs:7,8);
      \addplot [fill=red!20, draw=none, domain=0.05:7, samples=200] 
      ({x}, { (x^2 - 4)/(2.5*x) }) 
      -- (axis cs:7,8) -- (axis cs:0.05,8) -- cycle;
      \addplot [thick, red, domain=0.05:7, samples=200] 
      ({x}, { (x^2 - 4)/(2.5*x) });
      \draw[red, thick] (axis cs:0,-4) -- (axis cs:0,8);
      \end{axis}
  \end{tikzpicture}
  \caption{The set $\mathcal S_{a,b}$ for $a=2.5$ and $b=4$.}
  \label{fig:data:sets1}
 \end{minipage}
 \hfill
 \begin{minipage}{0.48\textwidth}
  \centering
  \begin{tikzpicture}[scale=0.85]
      \begin{axis}[
      axis lines = middle,
      xlabel = {$u$},
      ylabel = {$r$},
      ymin = -4, ymax = 8,
      xmin = -1, xmax = 7,
      xtick = {1,2,3,4,5,6},
      ytick = {-3,-2,-1,1,2,3,4,5,6,7},
      axis equal image,
      inner axis line style={-stealth},
      axis on top,
      ]
      \clip (axis cs:0,-4) rectangle (axis cs:7,8);
      \addplot [name path=sup, thick, blue, domain=0:7, samples=200] 
      ({x}, { 1.25*x + 0.5*sqrt(2.25*x^2 + 16) });
      \addplot [name path=inf, thick, blue, domain=0:7, samples=200] 
      ({x}, { 1.25*x - 0.5*sqrt(2.25*x^2 + 16) });
      \addplot [blue!20] fill between [of=sup and inf];
      \draw[blue, thick] (axis cs:0,-2) -- (axis cs:0,2);
      \end{axis}
  \end{tikzpicture}
  \caption{The set $\Sigma_{a,b}$ for $a=2.5$ and $b=4$.}
  \label{fig:data:sets2}
 \end{minipage}
\end{figure}

With this coercivity result in mind, let us state our main existence result.
\begin{theorem} \label{thm:ex_odds}
    If the data set $\D$ satisfies \eqref{eq:condD} then there exists an optimal data-driven solution $(u,r)\in \cC$ of \eqref{eq:main}. It is given as a minimizer of the functional $\scm I+\iota_R$ on $V$.
\end{theorem}

The proof of the main existence result appears at the end of this section. It requires a number of preliminary steps. The first is the weak-weak$^\ast$ closedness of the constraint $\cC$.

\begin{lemma} \label{lemma_closed}
 The differential constraint $\cC$ is weak-weak$^\star$ closed in $H^1(\Omega)\times L^\infty(\Omega)$. 
\end{lemma}

\begin{proof}
 Let $(u_n,r_n)\subset\cC$ so that
 \begin{equation} \label{eq:weak_n}
 \int_\Omega \nabla u_n \nabla \psi \dd x = \int_\Omega r_n u_n \psi \dd x,
 \end{equation}
 for all $\psi\in C^\infty(\Omega)$. If we suppose that $u_n\rightharpoonup u$ in $H^1(\Omega)$ and $r_n\stackrel{*}{\rightharpoonup}r$ in $L^\infty(\Omega)$ for some $(u,r)\in V$ then $\|r_n\|_\infty\le C$ with $C>0$, $\nabla u_n\rightharpoonup \nabla u$ in $L^2(\Omega)$ and, up to a subsequence, $u_n\to u$ strongly in $L^2(\Omega)$ and a.e. in $\Omega$ so that $u\ge 0$. Moreover, $u_n r_n\rightharpoonup ur$ in $L^2(\Omega)$ since for all $\psi\in C^\infty(\Omega)$ we have that 
 \begin{equation}\label{eq_conv01}\begin{aligned}
 \int_\Omega (u_n r_n-ur ) \psi\dd x&=\int_\Omega (u_n-u) r_n \psi+u(r_n-r)\psi\dd x\\&\leq C \|u_n-u\|_2\|\psi\|_2 +\int_\Omega(r_n-r)u\psi\dd x\to 0 
 \end{aligned}\end{equation}
 as $n\to\infty$ since $u\psi\in L^1(\Omega)$. This allows us to pass to the limit $n\to \infty$ in \eqref{eq:weak_n} and get that $(u,r)$ is a weak solution to $-\Delta u = ru$ in $\Omega$, $\partial_\nu u=0$ on $\partial\Omega$ with $u\geq 0$. 
\end{proof}

Next we consider another way of characterizing the coercivity conditions \eqref{eq:condD}, \eqref{eq:condD_new}.  

\begin{lemma}\label{lemma:dist}
Let $\D$ be a given data set. 
\begin{itemize} 
\item[(a)] The following are equivalent:
\begin{enumerate}
    \item[(i)] There exist $c_1,c_2>0$ and $\gamma \in \R$ such that 
\begin{align}\label{eq:2_coercive}
 \operatorname{dist}\left((x,u,r), \D \right)\ge c_1 u^2 -c_2- \gamma ur \quad \text{for all} \,\, (x,u,r)\in Y.
\end{align}
\item[(ii)] \eqref{eq:condD} holds for some $a\in\R, b>0$, i.e.,  $\D \subseteq \Omega\times  \mathcal S_{a,b}$.
   \end{enumerate}
\item[(b)] The following are equivalent:
\begin{enumerate}
    \item[(iii)] There exist $c_1,c_2>0$ and $\gamma \in \R$ such that 
\begin{align}\label{eq:2_coercive_new}
 \operatorname{dist}\left((x,u,r), \D \right)\ge c_1 (u^2+r^2) -c_2- \gamma ur \quad \text{for all} \,\, (x,u,r)\in Y.
\end{align}
\item[(iv)] \eqref{eq:condD_new} holds for some $a\in\R, b>0$, i.e. $\D\subseteq \Omega\times \Sigma_{a,b}$.
   \end{enumerate}
\end{itemize}
\end{lemma}

\begin{proof} Part (a):
(i) $\Rightarrow$ (ii): Suppose that \eqref{eq:2_coercive} holds. Then, for all $(x,u,r)\in\D$ we have $0\ge c_1 u^2 -c_2- \gamma ur$ and hence $u^2\le aur +b$ with $a= \gamma/c_1\in\R$ and $b= c_2/c_1>0$.

\medskip

(ii) $\Rightarrow$ (i): Consider the set $\D'\coloneqq \Omega\times \mathcal S_{a,b}$. Since $\D\subseteq \D'$ we get that $\dist(\cdot,\D)\geq \dist(\cdot, \D')$ and hence it is sufficient to establish \eqref{eq:2_coercive} for $\D'$. Moreover, since $\dist((x, u, r), \D')\ge \dist((u, r), \mathcal S_{a,b})$ it is sufficient to establish the lower estimate 
\begin{equation} \label{eq:dist_hom}
\dist((u, r), \mathcal S_{a,b}) \geq c_1 u^2-c_2-\gamma ur \mbox{ for all } (u,r)\in [0,\infty)\times\R.
\end{equation}
As in \cite[Lemma 5.9]{LSS23} we can suppose $b=0$. 
Since $(u,r)\in \mathcal S_{a,0}$ implies $(\lambda u, \lambda r)\in  \mathcal S_{a,0}$ for all $\lambda>0$ we see that the distance function is homogeneous, that is 
\begin{equation}\label{eq_hom}
\dist\left( (\lambda u,\lambda r), \mathcal S_{a,0} \right) = \lambda^2 \dist\left( (u,r), \mathcal S_{a,0} \right).
\end{equation}
Let $S^+=\{u^2+r^2 = 1, u\geq 0\}$ be the right half of the unit sphere. The homogeneity observation \eqref{eq_hom} implies that it is enough to establish \eqref{eq:dist_hom} on $S^+$. Since the set
$$
S^+\cap E \mbox{ with } E \coloneqq \{u^2 \geq 2aur\}
$$
is compact and has positive distance to $ \mathcal S_{a,0}$, there exists $\eta>0$ such that
$$
\dist((u, r), \mathcal S_{a,0}) \geq \eta \mbox{ for all } (u,r)\in S^+\cap E.
$$
With $M\coloneq \max _{(\tilde u, \tilde r) \in E}\left(\tilde u^2-2a \tilde u\tilde r\right)$ this implies for $(u,r)\in S^+\cap E$ that
\begin{equation}\label{boundduv}
 \dist((u, r),  \mathcal S_{a,0}) \ge \eta \ge \frac{\eta}{M}\left(u^2-2aur\right).
\end{equation}
This also holds for $(u,r)\in S^+\cap E^c$ since then $u^2-2aur< 0$. This finishes the proof of Part (a).

\medskip
Part (b): Up to some minor details, the proof is the same as for Part (a).
\end{proof}

Thanks to the previous lemma we can prove Lemma~\ref{l:coercivity_omega}.
\begin{proof}[Proof of Lemma~\ref{l:coercivity_omega}]
Part (a):
Clearly, since $\cC$ is weak-weak$^\ast$ closed, the functional $\tilde I\coloneqq \infty\chi_{V\setminus\cC}$ is weak-weak$^\ast$ lower-semicontinuous and below $I$. By the alternative definition \eqref{def:hatsc-} of $\scm I$  by Lemma~\ref{lemma_sc-}, we find $\scm I\geq \tilde I$ so that in particular $\scm I=\infty$ on $V\setminus \cC$. Next consider $(u,r)\in \cC$.
Then, by Lemma~\ref{lemma:dist}(a), $\dist(\cdot, \D)$ satisfies \eqref{eq:2_coercive} so that 
 $$
 I(u,r) = \int_\Omega \dist((x,u,r), \D )\dd x \ge \int_\Omega \left( c_1 u^2- c_2 -\gamma u r \right)\dd x = c_1\|u\|_2^2 -c_2|\Omega|,
 $$
 because $(u,r)$ satisfies \eqref{eq:main}. Let us check that also $\scm I$ satisfies the same lower bound. The functional $\check I(u,r) \coloneqq c_1\|u\|_2^2 -c_2|\Omega|$ for $(u,r)\in V$, which is a lower bound for $I$, is also weak-weak$^\ast$ lower semi-continuous since $H^1(\Omega)$ compactly embedds into $L^1(\Omega)$ and $L^\infty(\Omega)=L^1(\Omega)^\ast$. Now we use again the alternative definition \eqref{def:hatsc-} of $\scm I$ and Lemma~\ref{lemma_sc-} to get 
 $$
 \scm I(u,r)+\iota_R(r) \geq c_1\|u\|_2^2 -c_2|\Omega|+ \iota_R(r)
 $$
 for all $(u,r)\in V$. Coercivity in unbounded $r$-directions stems from $\iota_R$. For sequences $(u_n,r_n)$ in $\cC$ with $\|r_n\|_\infty\leq R$, \eqref{eq:main} implies $\|\nabla u_n\|_2^2 \leq R\|u_n\|_2^2$ so that $L^2$-coercivity in the $u$-direction implies $H^1$-coercivity.
 \medskip
 
Part (b): It suffices to consider $(u,r)\in \mathcal{C}$. Similarly to part (a),  \eqref{eq:2_coercive_new} from Lemma~\ref{lemma:dist}(b) yields
$$
I(u,r) \ge \int_\Omega \left( c_1 (u^2 + r^2)- c_2 -\gamma u r \right)\dd x = c_1(\|u\|_2^2 + \|r\|_2^2)-c_2|\Omega|
$$
form which we find the $H^1$-$L^2$-coercivity of $I$ as before.
Since the lower bound $(u,r) \mapsto c_1\|u\|_2^2 -c_2|\Omega|$ of $I$ is weak-weak$^\ast$ lower semi-continuous this also implies the $H^1$-$L^2$-coercivity of $\scm I$.
\end{proof}

We can finally prove the main result of this section.
\begin{proof}[Proof of Theorem~\ref{thm:ex_odds}]
Let us consider a minimizing sequence $(u_n,r_n) \subset \cC$ for $m\coloneqq \inf_V (\scm I + \iota_R)=\inf_{\mathcal C} (\scm I+\iota_R)$. By the coercivity result of Lemma~\ref{l:coercivity_omega}{(a)} the sequence $(u_n)$ is bounded in $H^1(\Omega)$ and $\|r_n\|_\infty\leq R$. So, up to subsequences $u_n\rightharpoonup u$ in $H^1(\Omega)$ and $r_n\stackrel{*}{\rightharpoonup} r$ in $L^\infty(\Omega)$ by Banach-Alaoglu Theorem (see e.g. \cite[Theorem 3.16]{brezis}), for some $(u,r)\in H^1(\Omega)\times L^\infty(\Omega)$ with $\|r\|_\infty\leq \liminf_n \|r_n\|_\infty\leq R$. Since  $\mathcal{C}$ is weak-weak$^\ast$ closed by Lemma~ \ref{lemma_closed} we have $(u,r)\in\mathcal{C}$ and the functional $\iota_R$ vanishes both on the sequence $r_n$ and its weak$^\ast$-limit $r$. By the weak-weak$^*$ lower semi-continuity of $\scm I$ from Lemma~\ref{lemma_sc-} we see that $(u,r)$ is a minimizer of $\scm I+\iota_R$. 
\end{proof}

\section{Consistency}\label{sec_consistency}
Let us suppose that the data set $\D$ stems from a given reproduction law $(x, u)\mapsto \varrho(x,u)$ and that it contains all possible data points, i.e.,
\begin{equation}
\label{eq:def_D}
\D = \{(x,u,\varrho(x,u)): x\in\Omega, u \in [0,\infty)\}.
\end{equation}
We can now consider optimal data-driven solutions of \eqref{eq:main} for this data set and we can also consider weak solutions $U\in H^1(\Omega)$ of 
\begin{equation} \label{eq:main2} 
-\Delta U= \varrho(x,U) U \mbox{ in } \Omega, \quad \partial_\nu U =0, \mbox{ on } \partial \Omega.
\end{equation}
The question is: under what condition on $\varrho$ can we say that every optimal data-driven solution of \eqref{eq:main} also satisfies \eqref{eq:main2}? The answer is given by the following theorem.

\begin{theorem} \label{thm:consistency} Let $\Omega\subset \R^N$ be a bounded convex domain and let $R>0$ be the cut-off value in the penalty functional $\iota_R$. Suppose $\varrho: \Omega\times [0,\infty)\to \R$ is a $C^2$-function such that $\sup_{x\in \Omega}|\varrho(x,0)|\leq R$, $\|D^2\varrho\|\leq \frac{1}{2R}$  for any (sub-multiplicative) matrix norm and one of the following three conditions holds:
\begin{itemize}
    \item[(a)] $\|\varrho\|_\infty \leq R$,
    \item[(b)] $\varrho$ is concave and $\|\varrho^+\|_\infty \leq R$,
    \item[(c)] $\varrho$ is convex and $\|\varrho^-\|_\infty \leq R$.
\end{itemize}
If the data set is given by \eqref{eq:def_D} then any optimal data-driven solution $(u,r)$ is such that $u$ is also a weak solution of \eqref{eq:main2}. Reversely, if $u$ is a weak solution of \eqref{eq:main2} and $\sup_{x\in \Omega} |\varrho(x,u(x))|\leq R$ then the pair $\left(u(\cdot),\varrho(\cdot, u(\cdot))\right)$ is an optimal data-driven solution of \eqref{eq:main}.
\end{theorem}

This result follows directly from the next theorem through the equivalence of (i) and (iii).

\begin{theorem} \label{thm:character_minIstar}
Under the assumptions of Theorem~\ref{thm:consistency} the following three statements are equivalent:
\begin{enumerate}
\item[(i)] $u$ solves \eqref{eq:main2} and 
$\sup_{x\in \Omega} |\varrho(x)|\leq R$ where $\varrho(x)\coloneqq \varrho(x,u(x))$
\item[(ii)] $\scm I(u,\varrho)+\iota_R(\varrho)=0$ 
\item[(iii)] $(u,\varrho)=\argmin \,(\scm I+\iota_R)$ 
\end{enumerate}
\end{theorem}

The proof of Theorem~\ref{thm:character_minIstar} is done as follows: first we show in Lemma~\ref{lem:marginal_convexity} that $\dist((x,u,r),\D)=\conv_r \dist((x,u,r),\D)$ under the assumptions on $\varrho$ of Theorem~\ref{thm:consistency}. This implies that $I=I^\ast$ on $\{(u,r)\in \cC: \|r\|_\infty\leq R\}$. The proof of Theorem~\ref{thm:character_minIstar} is then concluded by applying the following weaker form of Theorem~\ref{thm:character_minIstar} together with the result $I^\ast=\scm I$ on $\cC$ from Theorem~\ref{thm:main_Istar_scm_I}. 

\begin{theorem} \label{thm:character_minI}
Suppose that the data set is given by \eqref{eq:def_D}. Then the following three statements are equivalent:
\begin{enumerate}
\item[(i)] $u$ solves \eqref{eq:main2} and
$\sup_{x\in \Omega} |\varrho(x)|\leq R$ where $\varrho(x)\coloneqq \varrho(x,u(x))$
\item[(ii)] $I(u,\varrho)+\iota_R(\varrho)=0$
\item[(iii)] $(u,\varrho)=\argmin\,(I+\iota_R)$
\end{enumerate}
\end{theorem}

\begin{remark}
Notice that in contrast to Theorem~\ref{thm:character_minIstar} no further assumption on the nature of the known reproduction laws $\varrho$ is needed for Theorem~\ref{thm:character_minI}. However, our main result of Theorem~\ref{thm:consistency} does not follow from Theorem~\ref{thm:character_minI} because optimal data-driven solutions are characterized as minimizers of $\scm I+\iota_R$ and not as minimizers of $I+\iota_R$, and even though $\inf_V (I+\iota_R)=\inf_V (\scm I+\iota_R)$, a minimizer of $\scm I+\iota_R$ is not necessarily a minimizer of $I+\iota_R$. 
\end{remark}

\begin{proof}[Proof of Theorem~\ref{thm:character_minI}] We consider the implications separately. 

\noindent
 (i) $\to$ (ii): If $u$ solves \eqref{eq:main2} and if we set $\varrho(x)\coloneqq \varrho(x, u(x))$ then $(x,u(x),\varrho(x))\in\D$ for all $x\in\Omega$ so that $\dist( (x,u(x),\varrho(x)), \D)=0$ for all $x\in \Omega$ and $\iota_R(\varrho)=0$ which implies (ii).

\smallskip

\noindent
(ii) $\to$ (i): Clearly, if $I(u,\varrho)+\iota_R(\varrho)=0$ then both $I(u,\varrho)$ and $\iota_R(\varrho)$ vanish. In particular, $\|\varrho\|_\infty\leq R$ and $\dist((x,u(x),\varrho(x)), \D)=0$ for all $x\in\Omega$, that is $(x, u(x), \varrho(x))\in\D$ for all $x\in\Omega$. The definition of the data set $\D$ implies $\varrho(x)=\varrho(x,u(x))$. Moreover $0=I(u,\varrho)<\infty$ implies that  $(u,\varrho)=(u,\varrho(\cdot, u(\cdot))) \in\mathcal{C}$ so that $u$ solves \eqref{eq:main2}.

\smallskip

\noindent
(ii) $\to$ (iii): Since $I+\iota_R\ge 0$ and $I(u,\varrho)+\iota_R(\varrho)=0$ we obtain (iii).

\smallskip

\noindent
(iii) $\to$ (ii): Suppose $(u, \varrho)$ is a minimizer of $I+\iota_R$ over $V$. Then $\iota_R(\varrho)=0$. Let $\varrho_0(x) \coloneqq \varrho(x,0)$. Then $(0,\varrho_0)\in\cC$, $\dist((x,0, \varrho_0(x)),\D)=0$, and, by assumption, $\|\varrho_0\|_\infty\leq R$. Therefore we find the inequality $0\leq I(u,\varrho)+\iota_R(\varrho)\leq I(0,\varrho_0)+\iota_R(\varrho_0)=0$ which implies (ii). 
\end{proof}

\begin{lemma} \label{lem:marginal_convexity} Suppose $\Omega\subset \R^N$ is a bounded convex domain and that the $C^2$-function $\varrho: \Omega\times [0,\infty)$ satisfies the assumptions from Theorem~\ref{thm:consistency}. Then the map $r \mapsto \dist((\cdot,\cdot,r),\D)$ is convex and therefore $I=I^\ast$ on $\{(u,r)\in \cC: \|r\|_\infty \leq R\}$.
\end{lemma}

\begin{proof}
Fix $(x,u)\in \Omega\times [0,\infty)$ and consider the map $\tau: \Omega\times [-R,R] \times[0,\infty)$ given by $\tau(\tilde x,r,\tilde u)\coloneqq |x-\tilde x|^2+|u-\tilde u|^2 + |r-\varrho(\tilde x,\tilde u)|^2$. Since the data set $\D$ is given by \eqref{eq:def_D} we see that 
$$
\dist((x,u,r),\D) = \inf_{\tilde x\in\Omega, \tilde u\ge 0} \tau(\tilde x,r,\tilde u).
$$
Note that we have frozen the variables $x$ and $u$ and consider only the $r$-dependence. We need to check whether $r \mapsto \dist((\cdot,\cdot,r),\D)$ is convex. Since $(\tilde x,\tilde u)$ vary in a convex domain $\Omega\times [0,\infty)$, this is a case of marginal convexity for which Theorem~5.7 in \cite{Rockafellar} and the subsequent consideration applies. It states that a sufficient condition for the convexity of $r \mapsto \dist((\cdot,\cdot,r),\D)$ is the convexity of $(\tilde x, \tilde u,r)\mapsto \tau(\tilde x, \tilde u,r)$. For this it is sufficient to show that the Hessian matrix $D^2\tau $ of $\tau$ is positive semi-definite, which is given by 
\begin{equation}\label{H:ndim} D^2\tau = 2 
\begin{pmatrix}
\Id + \partial_{\tilde x}\varrho \partial_{\tilde x}\varrho^{\mathrm{T}}+ (\varrho-r)D^2_{\tilde x}\varrho & - \partial_{\tilde x}\varrho &  \partial_{\tilde u}\varrho\partial_{\tilde x}\varrho + (\varrho-r)\partial^2_{\tilde x,\tilde u}\varrho \\
- \partial^{\mathrm{T}}_{\tilde x}\varrho & 1 & -  \partial_{\tilde u}\varrho \\
\partial_{\tilde u}\varrho\partial_{\tilde x}\varrho^{\mathrm{T}} + (\varrho-r)\partial^2_{\tilde x,\tilde u}\varrho^{\mathrm{T}} & - \partial_{\tilde u}\varrho & (\partial_{\tilde u}\varrho)^2 + (\varrho-r) \partial^2_{\tilde u}\varrho +1
\end{pmatrix},\end{equation}
 where $D^2_{\tilde x}\varrho$ is the Hessian matrix of $\varrho$ w.r.t. the variable $\tilde x$.
In order to prove positive semi-definiteness of $D^2\tau$ we need to check that all principal minors are non-negative. Starting from the north-west corner of the matrix, we see that it is enough to have 
 \begin{equation}\label{cond:convex}
\Id + \partial_{\tilde x} \varrho \partial_{\tilde x}\varrho^T+(\varrho-r)D^2_{\tilde x} \varrho\geq 0,  \quad \det(\Id + (\varrho-r)D^2_{\tilde x}\varrho)\ge 0, \quad \det D^2 \tau\ge 0.
\end{equation}
By the Gauss reduction technique we get that
\begin{align*} \det D^2\tau & = 2^{N+2}\det 
\begin{pmatrix}
\Id+(\varrho-r)D^2_{\tilde x}\varrho & 0 & (\varrho-r)\partial^2_{\tilde x,\tilde u}\varrho \\
0 & 1 & 0 \\
(\varrho-r)\partial^2_{\tilde x,\tilde u}\varrho^{\mathrm{T}} & 0 & (\varrho-r) \partial^2_{\tilde u}\varrho +1
\end{pmatrix} \\
& = 2^{N+2} \det \begin{pmatrix}
\Id+(\varrho-r)D^2_{\tilde x}\varrho & (\varrho-r)\partial^2_{\tilde x,\tilde u}\varrho \\
(\varrho-r)\partial^2_{\tilde x,\tilde u}\varrho^{\mathrm{T}} & (\varrho-r) \partial^2_{\tilde u}\varrho +1
\end{pmatrix} \\
& = 2^{N+2} \det (\Id + (\varrho-r) D^2 \varrho).
\end{align*}
 In view of \eqref{cond:convex} the map $\tau$ is convex if the matrix $\Id+(\varrho-r) D^2\varrho$ is positive semi-definite. By the Rayleigh quotient characterization of the lowest eigenvalue, one finds for real symmetric matrices $A, B$ Weyl's bound $\lambda_1(A+B)\geq \lambda_1(A)+\lambda_1(B)$. Therefore, for the positive definiteness of $\Id + (\varrho-r)D^2\varrho$ it is sufficient to have 
\begin{equation} \label{eq:ineq_lambda1}
1 \geq -\lambda_1 \Bigl((\varrho-r) D^2\varrho\Bigr)= \lambda_{\max} \Bigl((r-\varrho) D^2\varrho\Bigr). 
\end{equation}
Under assumption (a) of the theorem we have $|r-\varrho|\leq 2R$ and a sufficient condition for \eqref{eq:ineq_lambda1} is that the spectral radius $\sigma_\text{rad}(D^2\varrho)$ is smaller or equal than $\nicefrac{1}{2R}$. Since any sub-multiplicative matrix norm is an upper bound for the spectral radius, the claim follows.

\medskip
Under assumption (b) we have that $D^2\varrho\leq 0$ is negative semi-definite. Therefore, since $\|\varrho^+\|_\infty \leq R$ and $r\in[-R,R]$ we have that 
$$
\lambda_{\max}((r-\varrho) D^2\varrho) \leq \lambda_{\max}((r-\varrho^+) D^2\varrho)\leq 2R \sigma_\text{rad}(D^2\varrho)
$$
while under assumption (c) with $D^2\varrho\geq 0$, $\|\varrho^-\|_\infty\leq R$ we have 
$$
\lambda_{\max}((r-\varrho) D^2\varrho) \leq \lambda_{\max}((r+\varrho^-) D^2\varrho)\leq 2R \sigma_\text{rad}(D^2\varrho).
$$
In both cases the claim follows as in case (a).
\end{proof}

\begin{remark}\label{rem:models}
As an example for \eqref{eq:main2} let us consider the diffusive logistic equation
\begin{equation} \label{eq:mainMod} \left\{
 \begin{alignedat}{2}
-\Delta u &= a(b-u) u, \quad &&\text{in}~ \Omega,\\
\partial_\nu u &=0, \quad &&\text{on}~ \partial \Omega,\\
u&>0, \quad &&\text{in}~ \Omega,
\end{alignedat}{} \right.
\end{equation}
with $a, b>0$, cf. \cite{O92}. By \cite[Theorem 3]{O92}, there exists for any $a, b>0$ a unique solution $u$ to \eqref{eq:mainMod} which (by the simple nonlinearity) is given by $u\equiv b$ on $\Omega$. Thus, the model \eqref{eq:mainMod} is of limited interest. However, since $\varrho(u)=a(b-u)$ is concave for $u\in [0,\infty)$ the assumption of Theorem~\ref{thm:consistency}(b) is fulfilled provided $R\geq ab=\|\varrho^+\|_\infty$. 

\medskip

As a more general model, we consider the diffusive logistic equation
with spatial heterogeneity of the type  
\begin{equation}\label{xs:prob}
   \left\{
 \begin{alignedat}{2}
        -\Delta u &= a(x)(b(x)-u)u,\quad  && \text { in } \Omega, \\
        \partial_\nu u&=0 \quad && \text { on } \partial \Omega,
 \end{alignedat}{} \right.
\end{equation}
with positive weights $a,b \in C^{2}(\Omega)$. Existence and qualitative properties of solutions to \eqref{xs:prob}, under suitable assumptions on the weights $a,b$, were studied in \cite{Du2006,   Henry1981, CC2003, XS23}.
Now we will state conditions on $a,b$ such that Theorem \ref{thm:consistency} is satisfied.
Let
$$\varrho(x,u)\coloneqq a(x)b(x)-a(x) u\in C^2(\Omega\times [0,\infty))$$
and 
$$
D^2\varrho=
    \begin{pmatrix}
        &  D^2(a(x)b(x))-D^2a(x) u & -  \nabla a(x)
            \\
        & - \nabla a(x)^T & 0 
    \end{pmatrix}.
$$
Since $u$ is not bounded from above, it is difficult to prove (a) and (c) of Theorem \ref{thm:consistency}. So we will focus on (b). On the one hand, by the structure of $D^2\varrho$, in order to make it negative semi-definite, we need to assume $a(x)\equiv a\in \R^+$ and that $b$ is a concave function.
On the other hand, 
$$
\varrho^+(x,u)=a\max\{0,b(x)- u\}
$$
so that we need to assume $\|\varrho^+\|_\infty=a\|b\|_\infty\le R$. Moreover, since $\Omega$ is a bounded domain and $b \in C^{2}(\Omega)$, we see that $\|D^2\varrho\|\le \frac{1}{2R}$ if and only if $\|D^2b\|\le \frac{1}{2R}$.
Thus in summary, Theorem~\ref{thm:consistency}(b) holds provided $b$ is concave, $a\|b\|_{\infty}\le R$ and $\|D^2b\|\le \frac{1}{2R}$.
\end{remark}

\section{Convergence}\label{sec_conv}
{
The convergence of data sets and their associated functionals as well as their minimizers (i.e. optimal data-driven solutions) is intimately connected with the notion of $\Gamma$-convergence. We recall the definition and some fundamental properties.

\begin{definition}[$\Gamma$-convergence on metric spaces]\label{def:Gamma}
Let $(X,d)$ be a metric space. A sequence of functionals $F_n\colon X\to [-\infty, \infty]$ $\Gamma$-converges to $F\colon X\to [-\infty, \infty]$ if:
\begin{itemize}
 \item For all $x\in X$ and for all sequences $x_n\to x$ we have
 $$
 F(x)\le \liminf_{n\to\infty} F_n(x_n).
 $$
 \item For all $x\in X$ there exists a sequence $x_n\to x$ such that
 $$
 F(x) \ge \limsup_{n\to\infty} F_n(x_n).
 $$
\end{itemize} 
\end{definition}
One of the most important features of $\Gamma$-convergence is the preservation of minimizers: if for each $n\in \N$ the point $x_n$ is a minimizer of $F_n$ and $(F_n)$ $\Gamma-$converges to $F$ then any accumulation point $x$ of $(x_n)$ is a minimizer of $F$, see cf. \cite{DalMaso}.

\medskip

\noindent
{\bf Topological vs. sequential $\Gamma$-convergence.} The notion of $\Gamma$-convergence on topological spaces is less simple than on metric spaces (we do not give the definition here; for details cf. \cite{DalMaso}). In our setting we would like to discuss $\Gamma$-limits for functionals on $V=H^1(\Omega)\times L^\infty(\Omega)$ equipped with the weak-weak$^\ast$ topology. While on the entire infinite dimensional Banach space $V$ the weak-weak$^\ast$ topology is not metrizable, it is metrizable (\cite[Theorem 3.29]{brezis}) on bounded subsets of $V$ since $H^1(\Omega)$ is a reflexive Sobolev space and $L^\infty(\Omega)$ is the dual of the separable space $L^1(\Omega)$. Therefore, if the functionals $(F_n), F$ are equi-coercive, i.e., 
 \begin{equation} \label{eq:babycoercive}
 \alpha(| x |) \le F_n(x), F(x)
 \end{equation}
 with a function $\alpha \colon [0,\infty) \to \R$ such that $\alpha(t) \to \infty$ as $t \to \infty$, then we can use Definition~\ref{def:Gamma} with $x_n\to x$ replaced by $u_n\rightharpoonup u$ in $H^1(\Omega)$ and $r_n \stackrel{\ast}{\rightharpoonup} r$ in $L^\infty(\Omega)$, see \cite[Lemma 2.4]{CMO}.

 \medskip

\noindent
{\bf Existence of $\Gamma$-limit for constant sequences.} It is not obvious that constant sequences admit a $\Gamma$-limit. In metric spaces, however, the constant sequence $F_n =F$ possesses a $\Gamma$-limit given by $\scm F$. In the situation of $V=H^1(\Omega)\times L^\infty(\Omega)$ equipped with the weak-weak$^\ast$ topology, the $\Gamma$-limit of the constant sequence $F$ exists and equals $\scm F$ provided $F$ is coercive.
}
\medskip

Next we give the definition of convergence of data sets. We use the following definition of the (asymmetric) directed Hausdorff pseudo-distance of two data sets $\mathcal A, \mathcal B\subseteq Y$ by 
$$
\dist(\mathcal A, \mathcal B) \coloneqq \sup_{z \in \mathcal A} \inf_{\tilde z \in \mathcal B} |z-\tilde z|^2 
$$
where we abbreviate $z=(x,u,r), \tilde z=(\tilde x, \tilde u, \tilde r)\in Y$.

\begin{definition}\label{def_conv_data_set}
Given a sequence $(\D_n)\subseteq Y$ of sets and $\D\subseteq Y$ we say that $\D_n\to \D$ as $n\to\infty$ if for all $S>0$
\begin{equation} \label{eq:def_convergence}
\lim_{n\to\infty}\dist(\D\cap B_{S}(0), \D_n) = 0 =\lim_{n\to\infty}\dist(\D_n\cap B_{S}(0), \D).
\end{equation}
\end{definition}

Let us consider $(\D_n)\subseteq Y$ and we define the cost functionals $I_n\colon V\to [0,\infty]$ as
\begin{align}\label{eq:I_n}
 I_n(u,r) \coloneqq 
 \begin{cases} 
 \int_{\Omega} \dist( (x, u(x),r(x)), \D_n) \dd x, & (u,r) \in \mathcal{C}, \\ 
 \infty, & \text{otherwise}.
 \end{cases}
\end{align}

The main result about convergence is given in the following theorem. 

\begin{theorem} \label{thm:gammaconv}
Let $(\D_n),\D \subset Y$ be closed, nonempty data sets and $a\in \R, b>0$ such that $\D, \D_n\subseteq \Omega\times \Sigma_{a,b}$ for all $n\in\N$ with $\Sigma_{a,b}$ as in \eqref{eq:condD_new}. If $\D_n \to \D$ as $n\to \infty$ in the sense of Definition~\ref{def_conv_data_set} then the following holds:  
\begin{itemize}
    \item[(i)] The functionals $I_n$ $\Gamma$-converge to $\scm I$ as $n\to \infty$ and the same holds for $\scm I_n$. 
    \item[(ii)] If $R>0$ is sufficiently large, then optimal data-driven solutions of \eqref{eq:main} for $\D_n$ converge to optimal data-driven solutions of \eqref{eq:main} for $\D$.
\end{itemize} 
\end{theorem}

Part (i) of Theorem~\ref{thm:gammaconv} follows from Theorem~\ref{theorem:equiint} below while part (ii) will be proven at the end of this section. We need some technical lemmas.
While proving the equi-integrability of sequences is a non-standard task in \cite{BFL}, in our case, every bounded sequence is equi-$L^2$-integrable. This allows for a smoother progression of the argument. 

\begin{lemma}\label{lemma_c}
 Any bounded sequence $(u_n,r_n)\subset\mathcal{C}$ is equi-$L^2$-integrable, that is for all $\varepsilon>0$ there exists $\delta>0$ such that for any measurable set $E$ with $|E|<\delta$ then
 \begin{equation*}
 \sup_{n\in\N} \int_E \left( u_n^2 + r_n^2 \right) \dd x <\varepsilon.
 \end{equation*}
\end{lemma}

\begin{proof} Take a bounded sequence $(u_n,r_n)\subset\mathcal{C}$. The equi-$L^2$-integrability of $r_n$ is trivial since $\sup_{n\in \N}\|r_n\|_\infty<\infty$. By the Sobolev embedding we have 
$$
\int_E |u_n|^2\dd x \le \left(\int_E |u_n|^{2q}\dd x\right)^\frac{1}{q} |E|^\frac{1}{q'} \leq C\|u_n\|_{H^1}^2 |E|^\frac{1}{q'} 
$$
where $q\in (1,\infty)$ for $N=1,2$ and $q\in (1,\frac{N}{N-2})$ for $N\geq 3$. This implies equi-$L^2$-integrability for the sequence $(u_n)$. 
\end{proof}

 In the following lemma, we prove a growth condition and continuity of the distance functions  $\dist(\cdot, \D_n)$ and $\dist(\cdot, \D)$. Moreover, we show that the convergence $\D_n \to \D$ from Definition~\ref{def_conv_data_set} ensures that $\dist(\cdot, \D_n)$ converges uniformly to $\dist(\cdot, \D)$ and the corresponding integral functionals $I_n$ converge to $I$ in a suitable sense.
 
\begin{lemma}\label{prop:fn:f}
Let $f_n(x,u,r)\coloneqq \dist((x,u,r),\D_n)$ and $f(x,u,r)\coloneqq \dist((x,u,r),\D)$. Under the assumptions of Theorem \ref{thm:gammaconv}, they have the following properties:
\begin{enumerate}[label=(\roman*)]
 \item There exists a constant $c>0$, such that for all $(x,u,r)\in Y$, we have 
 $$
 0 \le f_n(x,u,r),\,f(x,u,r) \le c(1 + |x|^2 + |u|^2 + |r|^2).
 $$
 \item $f_n$ and $f$ are uniformly continuous on bounded subsets of $Y$.
 \item $f_n$ converges uniformly to $f$ on compact subsets of $Y$. 
 \item $I_n$ converges uniformly to $I$ on equi-$L^2$-integrable subsets of $V$, i.e. for all equi-$L^2$-integrable sets $B \subset V$ and for all $\varepsilon>0$ there exists $n_{\varepsilon} \in \N$, such that for all $(u,r) \in B$ and all $n \ge n_{\varepsilon}$ we have
 $$
 \left \vert \int_{\Omega} f_n(x,u(x),r(x)) - f(x,u(x), r(x)) \dd x \right \vert \le \varepsilon.
 $$
\end{enumerate}
\end{lemma}

\begin{proof}
(i) {Since for any given $S>0$ the sequence $\dist(\D_n\cap B_S(0),\D)$ converges to $0$ we get in particular that the set $\D_n\cap B_S(0)$ is nonempty for large enough $n$. We can therefore choose $S>0$ such that $\D_n\cap B_S(0)\neq \emptyset$ for all $n\in \N$.} If we take $(\tilde x, \tilde u, \tilde r)\in \D_n\cap B_S(0)$ then we may estimate for any $(x,u,r)\in Y$
\begin{align*}
   \dist((x,u,r), \D_n) &\le |x-\tilde x|^2+|u-\tilde u|^2+|r-\tilde r|^2\le 2\left( |x|^2+|u|^2+|r|^2+S^2 \right)
   \\ &
   \le c(1+ |x|^2 + |u|^2 + |r|^2) 
\end{align*}
with a similar estimate for $\dist((x,u,r),\D)$.

\medskip

(ii) It follows since both $f_n$ and $f$ are uniformly Lipschitz continuous on bounded subsets of $Y$.

\medskip

(iii) Let $K \subset Y$ compact and fix $z \in K$. We divide the proof into two steps.

\medskip

{
{\it Step $1$.}
Since $\D$ is closed, we can find $\tilde{z}(z)  \in \mathcal{D}$ such that $f(z) = |z-\tilde z(z)|^2$.
By the boundedness of $K$, $\tilde{z}(z)$ is contained in a ball $B_{S_1}(0)$ for some $S_1>0$. By Definition~\ref{def_conv_data_set} there 
exists a sequence $a_n^{S_1}\to 0$ such that
$$
f_n(\tilde z(z))\le a_n^{S_1}.
$$
Again by closedness of $\D_n$, for all $n\in\N$ there exist  $z_n(z) \in \mathcal{D}_n$ such that $f_n(\tilde z(z)) = |\tilde{z}(z) - z_n(z)|^2$. Then it follows that
    \begin{align*}
    f_n(z) &\le |z - z_n(z)|^2 \le  (|z - \tilde{z}(z)| + |\tilde{z}(z) - z_n(z)|)^2\\
    &= f(z) + f_n(\tilde z(z)) + 2\sqrt{f(z)f_n(\tilde z(z))}\\
    &\le f(z) + a_n^{S_1}+ 2\sqrt{f(z)a_n^{S_1}}
    \end{align*}
    and by the compactness of $K$ this implies 
    \begin{equation}\label{eq_unif_bound1}
            f_n(z) \le f(z) + \sigma_n^{S_1} \quad \mbox{ for some $\sigma_n^{S_1} \to 0$ }
    \end{equation}
uniformly for $z\in K$ as $n\to\infty$. Note that \eqref{eq_unif_bound1} also implies that $f_n(z)$ is uniformly bounded for $z\in K$ and $n\in\N$.
    
{\it Step $2$.}
Since $\D_n$ are closed, we can find $w_n(z) \in \mathcal{D}_n$ satisfying $f_n(z) = |z - w_n(z)|^2$ for all $n\in\N$. Note that there exists $S_2>0$ such that
    $$
    |w_n(z)|\le |w_n(z)-z|+|z|=\sqrt{f_n(z)}+|z|\le S_2
    $$
by the boundedness of $K$ and by the uniform boundedness of $f_n$ on $K$ given by {\it Step $1$}. Thus, for all $n\in\N$ and $z\in K$ the points $w_n(z)$ are contained in the ball $B_{S_2}(0)$. By Definition~\ref{def_conv_data_set} there exists a sequence $b_n^{S_2}\to 0$ and by the closedness of $\D$ there exists $\hat{w}_n(z) \in \mathcal{D}$ such that $ f(w_n(z)) = |w_n(z) - \hat{w}_n(z)|^2 \le b_n^{S_2}$. Similarly as before
    \begin{equation*}
    f(z) \le |z - \hat{w}_n(z)|^2 \le (|z - w_n(z)| + |w_n(z) - \hat{w}_n(z)|)^2 \le  f_n(z) + b_n^{S_2}+ 2\sqrt{f_n(z) b_n^{S_2}}
    \end{equation*}
which implies 
    \begin{equation} \label{eq_unif_bound2}
    f(z) \le f_n(z) + \tau_n^{S_2} \mbox{ for some } \tau_n^{S_2} \to 0
    \end{equation} 
uniformly for $z\in K$ as $n\to\infty$.}

If we collect the results \eqref{eq_unif_bound1} from {\it Step $1$} and \eqref{eq_unif_bound2} from {\it Step $2$} then we obtain the claim $\sup_{z \in K} |f_n(z) - f(z)| \to 0$.

\medskip

(iv)
Fix $\varepsilon > 0$. Since $B \subset V = H^1(\Omega) \times L^\infty(\Omega)$ is equi-$L^2$-integrable, there exists a constant $C_B > 0$ such that $\|r\|_\infty \le C_B$ for all $(u,r) \in B$. Furthermore, the family of functions $\{ 1 + |\cdot|^2+ |u|^2 + |r| ^2 : (u,r) \in B \}$ is equi-integrable in $L^1(\Omega)$. So we can find a constant $M > 0$ and a compact subset $K \subset \Omega$ such that,  for all $(u,r) \in B$:
\begin{equation} \label{eq:equi}
\int_{E_M} (1 + |x|^2+|u(x)|^2 + |r(x)|^2) \dd x < \frac{\varepsilon}{4c},
\end{equation}
where $c$ is the growth constant from (i) and
\begin{equation*}
E_M \coloneqq \{ x \in \Omega : |u(x)| > M \} \cup (\Omega \setminus K)
\end{equation*}
which follows from Lemma \ref{lemma_c} and Chebyshev's inequality 
$$
|\{ x \in \Omega : |u(x)| > M \}|\le \frac{1}{M^2}\int_{\Omega}|u(x)|^2\dd x\le \frac{C}{M^2},
$$
and where $C>0$ is the uniform upper bound on the $L^2$-norm of $u$ for elements $(u,r)\in B$. On $\Omega\setminus E_M$ the points $(x, u(x), r(x))$ take values in 
\begin{equation*}
K_M \coloneqq K \times [0, M] \times [-C_B, C_B]
\end{equation*}
which is a compact subset of $Y$. By (iii), $f_n$ converges to $f$ uniformly on $K_M$; hence, there exists $n_\varepsilon \in \mathbb{N}$ such that for all $n \ge n_\varepsilon$:
\begin{equation*}
\sup_{z \in K_M} |f_n(z) - f(z)| < \frac{\varepsilon}{2|\Omega|}.
\end{equation*}
We now estimate by (i) and \eqref{eq:equi}
\begin{align*}
\int_{\Omega} &|f_n(x, u(x), r(x)) - f(x, u(x), r(x))| \dd x \\&= \int_{E_M} |f_n(x, u(x), r(x)) - f(x, u(x), r(x))| \dd x + \int_{\Omega \setminus E_M} |f_n(x, u(x), r(x)) - f(x, u(x), r(x))| \dd x \\
&\le 2c \int_{E_M} (1 +|x|^2+ |u(x)|^2 + |r(x)|^2) \dd x + \int_{\Omega \setminus E_M} \sup_{z \in K_M} |f_n(z) - f(z)| \dd x \\
&< 2c \left( \frac{\varepsilon}{4c} \right) + |\Omega| \frac{\varepsilon}{2|\Omega|} = \varepsilon
\end{align*}
and we conclude.
\end{proof}

\begin{remark}
Note that the closedness of the data sets in Definition \ref{def_conv_data_set} and Theorem \ref{thm:gammaconv} is primarily relevant and utilized in part (iii) of Lemma \ref{prop:fn:f}.
\end{remark}

Part (i) of Theorem \ref{thm:gammaconv} follows from the next theorem, which shows that the sequence of functionals $(I_n)$ admits a $\Gamma$-limit which coincides with $\scm I$.

\begin{theorem} \label{theorem:equiint} Under the assumptions of Theorem \ref{thm:gammaconv} the sequential $\Gamma$-limit of the constant sequence $I$ as well as the sequential $\Gamma$-limit of $I_n$ exist and we have
 $$
 \Gamma - \lim_{n\to\infty} I_n 
 = 
 \Gamma - \lim_{n\to\infty} I = \scm I. 
 $$
 \end{theorem}

\begin{proof} Since $I$ is coercive by Lemma~\ref{l:coercivity_omega}(b), 
the sequential $\Gamma$-limit of the constant sequence $I$ exists and coincides with $\scm I$  (as remarked at the beginning of this section). The remaining part of the proof about the sequential $\Gamma$-limit of the sequence $(I_n)$ is divided into two steps. 

\medskip

{\it Step 1.} First we prove that for all sequences $(u_n,r_n)$ in $V$ such that $u_n \rightharpoonup u$ in $H^1(\Omega)$ and $r_n\stackrel{*}{\rightharpoonup} r$ in $L^\infty(\Omega)$ we have that 
 \begin{equation}\label{eq_limsupinf_new}
 \limsup_{n\to\infty} I_n(u_n, r_n)= \limsup_{n\to\infty} I(u_n, r_n), \quad  \liminf_{n\to\infty} I_n(u_n, r_n)= \liminf_{n\to\infty} I(u_n, r_n). 
 \end{equation}
Since $(u_n,r_n)$ is bounded by weak-weak$^\ast$ convergence the sequence $(u_n,r_n)$ is equi-$L^2$-integrable by Lemma~\ref{lemma_c} and therefore  Lemma \ref{prop:fn:f}(iv) yields
\begin{equation}\label{eq_lim}
 \lim_{n\to\infty} \left|
 \int_{\Omega} f_n (x,u_n, r_n) -f (x,u_n, r_n) \dd x\right| = 0.
 \end{equation}
Next we consider a distinction into several cases. Case~1: $(u_n, r_n)$ has no subsequence in $\cC$. In this case all four values in \eqref{eq_limsupinf_new} are equal to $\infty$. Case~2: Let $(\tilde u_n, \tilde r_n)$ be the subsequence of $(u_n,r_n)$ which collects all elements in $\cC$ so that $\liminf_{n\to\infty} I(u_n, r_n)= \liminf_{n\to\infty} I(\tilde u_n, \tilde r_n)<\infty$ and $\liminf_{n\to\infty} I_n(u_n, r_n)= \liminf_{n\to\infty} I_n(\tilde u_n, \tilde r_n)<\infty$. They are equal because $\liminf_{n\to\infty} a_n=\liminf_{n\to\infty} b_n$ provided we know $\lim_{n\to\infty} a_n-b_n=0$, cf. \eqref{eq_lim}. Now we need to consider both $\limsup$-expressions. We do this in two further subcases. Case~2.1: $(u_n,r_n)$ has a subsequence in $V\setminus \cC$ . In this case both $\limsup$ are equal to $\infty$. Case~2.2: $(u_n,r_n)$ has no subsequence in $V\setminus \cC$. In this case both $\limsup$ are finite and equal as seen in Case~2 but with $\liminf$ instead of $\limsup$. 

\medskip
{\it Step 2 -- conclusion.} We already know that the sequential $\Gamma$-limit of $I$ exists, and hence coincides with $\scm I$. To establish that  $\Gamma - \lim_{n\to\infty} I_n=\scm I$ we need to show that the following holds true for all $(u,r) \in V$: 
\begin{enumerate}
 \item[(a)] Every sequence $(u_n, r_n) \subset \mathcal{C}$ with $u_n \rightharpoonup u$ in $H^1(\Omega)$ and $r_n\stackrel{*}{\rightharpoonup} r$ in $L^\infty(\Omega)$ satisfies $\scm I(u,r) \le \liminf_{n\to\infty} I_n (u_n, r_n)$.
 \item[(b)] There exists a sequence $(u_n, r_n) \subset \mathcal{C}$ with $u_n \rightharpoonup u$ in $H^1(\Omega)$ and $r_n\stackrel{*}{\rightharpoonup} r$ in $L^\infty(\Omega)$, such that $\operatorname{sc}^-I(u,r) \ge \limsup_{n\to\infty}  I_n (u_n, r_n)$.
\end{enumerate}
Part (a) is ensured because {\it Step 1} implies 
$$
 \liminf_{n\to\infty} I_n(u_n, r_n)= \liminf_{n\to\infty} I(u_n, r_n) \ge \operatorname{sc}^-I(u,r).
$$
On the other hand, since $\Gamma-\lim I = \scm I$ exists, there exists a sequence $(u_n,r_n)$ such that $u_n\rightharpoonup u$ in $H^1(\Omega)$, $r_n\stackrel{*}{\rightharpoonup} r$ and $\scm I(u,r)\ge \limsup I(u_n,r_n)$. Together with {\it Step 1} this implies
$$
 \scm I(u,r) \ge \limsup_{n\to\infty} I(u_n, r_n) = \limsup_{n\to\infty} I_n(u_n, r_n),
$$
and hence (b) is shown. This completes the proof of $\Gamma-\lim_{n\to\infty} I_n =\scm I$. The fact that also $\Gamma-\lim_{n\to\infty} \scm I_n=\scm I$ follows from \cite[Proposition~6.11]{DalMaso}. 
\end{proof}

To finish the proof of Theorem~\ref{thm:gammaconv}, it remains to prove the statement on the convergence of optimal data-driven solutions.

{
\begin{proof}[Proof of part (ii) of Theorem~\ref{thm:gammaconv}:] In general, $\scm I_n+\iota_R$ does not $\Gamma$-converge to $\scm I+\iota_R$ because $\Gamma$-convergence is not additive but super-additive. However, in the case of Theorem~\ref{thm:gammaconv}, the functionals $I_n, \scm I_n$, $I$, and $\scm I$ are equi-coercive (by Lemma~\ref{l:coercivity_omega}(b) and since $\D_n, \D\subseteq \Omega\times \Sigma_{a,b}$ for all $n\in \N$) so that potential minimizers of $\scm I_n$ and $\scm I$ lie in a bounded set $M\subset V$. If we choose $R>0$ so large that $M\subset B_R(0,0)$, where $B_R(0,0)$ should be considered with the $L^2(\Omega)\times L^\infty(\Omega)$ norms, then the set of minimizers of $\scm I_n$ coincides with the set of minimizers of $\scm I_n+\iota_R$ and likewise the set of minimizers of $\scm I$ coincides with the set of minimizers of $\scm I+\iota_R$.

\medskip
From these considerations, the weak convergence of minimizers of $\scm I_n+\iota_R$ (i.e. optimal data-driven solutions of \eqref{eq:main} for the data set $\D_n$) to minimizers of $\scm I+\iota_R$ (i.e. optimal data-driven solutions of \eqref{eq:main} for the data set $\D$) follows from the weak convergence of minimizers of $\scm I_n$ to minimizers of $\scm I$
which holds by  \cite[Corollary~7.20]{DalMaso} and the equi-coercivity of $\scm I_n$.  
\end{proof}
}

\section{Identification of the lower semi-continuous envelope $\scm I$} \label{sec_scmI=Istar}

In this section we prove Theorem~\ref{thm:main_Istar_scm_I}, i.e., we show that $\operatorname{sc}^-I$ { defined in \eqref{def:sc-}} coincides with $I^\ast$, {from Definition~\ref{def:conv}}. We follow some ideas given in \cite[page 557, Theorem 3.6]{BFL}.

\subsection{Identification of $\scm I$ as a Radon measure}\label{sec_id:1}

Since $\scm I$ is defined in \eqref{def:sc-} in an abstract way, the first goal is to show that $\scm I$ can be seen as a Radon measure which is absolutely continuous with respect to the Lebesgue measure. For this purpose we first need to extend the functional $I$ to all open subsets $A\subseteq \Omega$ as follows:
\begin{align}\label{eq:I_f_A}
    I_A(u,r) \coloneqq 
     \begin{cases} 
         \int_A \dist( (x,u(x),r(x)), \D) \dd x, & (u,r) \in \cC \\ 
         \infty, & \text{otherwise}.
     \end{cases}
\end{align}
This also leads to the extension of $\scm I$
\begin{align} \label{def:scm_I_A}
\scm I_A (u,r) \coloneqq \inf \Bigl\{& \liminf_{n\to\infty} I_A(u_n,r_n): (u_n, r_n) \in V, u_n \rightharpoonup u \text { in } H^1(\Omega), r_n\stackrel{*}{\rightharpoonup} r \text { in } L^{\infty}(\Omega) \Bigr\}.
\end{align}
We now prove for fixed $(u,r)\in \cC$ that $ A \mapsto \operatorname{sc}^- I_A$ coincides with a Radon measure (i.e., a regular measure that is finite on compact sets; see \cite[p.~212]{folland}) which is absolutely continuous with respect to the Lebesgue measure. The primary tool for this proof will be the De Giorgi–Letta Theorem (Theorem \ref{thm_degiorgi}).

\begin{theorem}\label{abs:cont}
Let $(u,r)\in \cC$. Then there exists a Radon measure $\lambda$ on $\Omega$ such that $\scm I_A(u, r)=\lambda(A)$ for all open subsets $A\subseteq\Omega$. Moreover, the Radon measure is absolutely continuous with respect to the Lebesgue measure so that $F\coloneqq\frac{d \lambda}{d \mathcal{L}^N}$ is the $L^1$-density of $\lambda$, i.e.,
\begin{equation}\label{eq_deriv}
\scm I_A(u,r) = \int_A F(x) \dd x 
\end{equation}
for all open subsets $A\subseteq \Omega$.
\end{theorem}

An important property of $\scm I(u,r)$  used frequently in the proof of Theorem~\ref{abs:cont}, is that it can be characterized by fixing $u$ (in some sense). This is the main message of the next lemma. In order not to distract the reader from the main flow of ideas, its proof is given in the Appendix.

\begin{lemma}\label{scm_alt_def}
For all $(u,r)\in \cC$ we have
$$
\scm I (u,r) = \inf \left\{\liminf_{n\to\infty} \int_\Omega \dist((x,u(x),r_n(x)),\D)\dd x: { (r_n) \subset L^\infty(\Omega)}, r_n\stackrel{*}{\rightharpoonup} r \text { in } L^{\infty}(\Omega) \right\}.
$$
\end{lemma}

Before giving the proof of the main result (Theorem~\ref{abs:cont}), we recall the De Giorgi--Letta criterion, see \cite[pag. 29]{chfo} and \cite{DeGiorgiLetta}, which consists of four conditions which guarantee that a map $\lambda$ defined on all open subsets of $\Omega$ extends as a regular measure to the Borel $\sigma$-algebra. Recall that a measure defined on the $\sigma$-algebra of all Borel sets is regular if every Borel set is both outer and inner regular, see \cite[Definition 2.15]{rudin}.

\begin{theorem}[De Giorgi--Letta]\label{thm_degiorgi}
Let $X$ be a metric space and let $\mathscr{A}$ be the collection of open subsets of $X$. Suppose that $\lambda : \mathscr{A} \to [0,\infty]$ is such that:
\begin{enumerate}[label=(\arabic*)]
 \item $\lambda(\emptyset) = 0$;
 \item for all $A_1, A_2 \in \mathscr{A}$,
 $$
 \lambda(A_1 \cup A_2) \le \lambda(A_1) + \lambda(A_2);
 $$
 \item if $A_1, A_2 \in \mathscr{A}$ with $A_1 \cap A_2 = \emptyset$, then
 $$
 \lambda(A_1 \cup A_2) \ge \lambda(A_1) + \lambda(A_2);
 $$
 \item \emph{(inner regularity)} for every $A \in \mathscr{A}$,
 $$
 \lambda(A) = \sup\{ \lambda(B) : B \subset \subset A \}.
 $$
\end{enumerate}
Then the extension of $\lambda$ to all subsets $E \subset X$ defined by
$$
 \lambda(E) = \inf\{ \lambda(A) : E \subset A,\ A \in \mathscr{A} \}
$$
is an outer measure, and its restriction to the Borel $\sigma$-algebra $\mathscr{B}(X)$ defines a measure.
\end{theorem}
\begin{proof}[Proof of Theorem~\ref{abs:cont}]
The goal is to apply Theorem \ref{thm_degiorgi} to $ A \mapsto \lambda(A)\coloneqq \scm I_A (u,r)$ where $A\subseteq \Omega$ is open. Clearly, (1) is satisfied. 

\medskip
{
To prove (2), fix $\eta>0$,  open subsets $A_1$ and $A_2$ of $\Omega$, and recall $f_u$ from \eqref{def:dist:fix}.
By Lemma~\ref{scm_alt_def} there are sequences
$$
(r_n), (w_n)\subset L^\infty(\Omega) \quad \mbox{ such that } 
r_n \stackrel{*}{\rightharpoonup} r,\,\,\, w_n \stackrel{*}{\rightharpoonup} r \, \text { in } L^\infty(\Omega)
$$
as $n\to\infty$ and 
\begin{equation}\label{cond:gdbc}
\begin{aligned}
& \lim _{n \to \infty} \int_{A_1} f_u\left(x, r_n(x)\right) \dd x \le \operatorname{sc}^-I_{A_1}(u, r)+\eta, \\
& \lim _{n \to \infty} \int_{A_2} f_u\left(x, w_n(x)\right) \dd x \le \operatorname{sc}^-I_{A_2}(u, r)+\eta.
\end{aligned}\end{equation}
If $V_n\in L^\infty(\Omega)$ is given as 
$$
V_n \coloneqq r_n \chi_{\Omega\setminus A_2}+ w_n \chi_{A_2},
$$
then $V_n\stackrel{\ast}{\rightharpoonup} r$ in $L^\infty(\Omega)$ as $n\to \infty$. Consequently, 
$$
\begin{aligned}
\scm I_{A_1\cup A_2}(u, r) \le & \liminf_{n \to \infty} \int_{A_1\cup A_2} f_u\left(x, V_n(x)\right) \dd x \\
\le & \limsup _{n \to \infty} \int_{A_1\setminus A_2} f_u\left(x, r_n(x)\right) \dd x+\limsup _{n \to \infty} \int_{A_2} f_u\left(x, w_n(x)\right) \dd x \\
\le & 2 \eta+\scm I_{A_1}(u, r)+\scm I_{A_2}(u, r).
\end{aligned}
$$
So, letting $\eta \to 0^{+}$ we have shown (2). 
}

\medskip
To prove (3), we take $A_1$ and $A_2$ as two disjoint open subsets of $\Omega$. Then we consider sequences $(u_n, r_n) \subset \cC$ such that $u_n\rightharpoonup u$ in $H^1(\Omega)$ and $r_n\stackrel{\ast}{\rightharpoonup} r$ in $L^\infty(\Omega)$. Therefore 
\begin{eqnarray*}
\lefteqn{\liminf_{n\to\infty} \int_{A_1\cup A_2} \dist\left((x,u_n(x), r_n(x)),\D\right) \dd x} \\
&=& \liminf_{n\to\infty} \left[ \int_{A_1} \dist\left((x,u_n(x), r_n(x)),\D\right) \dd x + \int_{A_2} \dist\left((x,u_n(x), r_n(x)),\D\right) \dd x\right] \\
 & \geq& \liminf_{n\to\infty} \int_{A_1} \dist\left((x,u_n(x), r_n(x)),\D\right) \dd x + \liminf_{n\to\infty} \int_{A_2} \dist\left((x,u_n(x), r_n(x)),\D\right) \dd x \\ 
 & \geq&  \scm I_{A_1}(u,r)+ \scm I_{A_2} (u,r)
\end{eqnarray*}
by the property of the liminf. If we take the inf over all sequences $(u_n, r_n)$ such that $u_n\rightharpoonup u$ in $H^1(\Omega)$ and $r_n\stackrel{\ast}{\rightharpoonup} r$ in $L^\infty(\Omega)$ then we get the desired result (3).

\medskip

We finally prove that $\lambda(\cdot)$ is inner regular, i.e., (4). We may assume $(u,r)\in \cC$ since otherwise $\lambda(A)=\infty$ for all { open subsets $A\subset \Omega$}. By the monotonicity of the integral and the definition of $\lambda$, it is clear that if $B \subset A$ then $\lambda(B) \le \lambda(A)$, hence $\sup_{B \subset\subset A} \lambda(B) \le \lambda(A)$. We need to prove the reverse inequality $\lambda(A) \le \sup_{B \subset\subset A} \lambda(B)$. We proceed as follows: given $\varepsilon>0$, we need to find $C=C_\varepsilon\subset\subset A$ such that:
$$
\lambda(A)\le \lambda(C)+\varepsilon.
$$
By the absolute continuity of the Lebesgue integral (see \cite[Proposition 16.3]{Secchi}), for any $\varepsilon> 0$, there exists an open set $C \subset\subset A$ such that
\begin{equation}
 \label{eq_int_coda}
 \int_{A \setminus C} \operatorname{dist}((x, u(x), r(x)), \D)\dd x < \frac{\varepsilon}{2}.
\end{equation}
By the definition \eqref{def:scm_I_A} of $\lambda(C)$, there exists a sequence $(u_n, r_n)$ such that $u_n \rightharpoonup u$ in $H^1(\Omega)$, $r_n \stackrel{*}{\rightharpoonup} r$ in $L^\infty(\Omega)$ and
\begin{equation}\label{lambdaC}
\lim_{n\to\infty} \int_C \dist((x, u_n(x), r_n(x)), \D)\dd x \le \lambda(C) + \frac{\varepsilon}{2}.
\end{equation}
Let $\varphi \in C_c^\infty(C)$ be a cut-off function such that $0 \le \varphi \le 1$. We define the new sequence $(\tilde u_n, \tilde r_n)\subset V$ as follows:
$$
\tilde u_n(x) \coloneqq \varphi(x) u_n(x) + (1 - \varphi(x)) u(x), \quad \tilde r_n(x) \coloneqq \begin{cases} r_n(x) & \text{if } x \in C, \\ r(x) & \text{if } x \in \Omega \setminus C. \end{cases}
$$
We find that $\tilde r_n \stackrel{*}{\rightharpoonup} r$ in $L^\infty(\Omega)$. Similarly, since $\nabla \tilde u_n = \varphi \nabla u_n + (1 - \varphi) \nabla u + \nabla \varphi (u_n - u)$ and $u_n\to u$ strongly in $L^2$ we see that $\tilde u_n \rightharpoonup u$ in $H^1(\Omega)$ as $n\to \infty$. Then, by Lemma~\ref{lem:I_compare} { and \eqref{eq_int_coda}}, we deduce
$$
\begin{aligned}
\int_A \dist((x, \tilde u_n, \tilde r_n), \D)\dd x &= \int_C \dist((x, \tilde u_n, r_n), \D)\dd x + \int_{A\setminus C} \dist((x, u, r), \D)\dd x \\
&\le \int_C \dist((x, u_n, r_n), \D)\dd x + o(1) + \frac{\varepsilon}{2}
\end{aligned}
$$
where $o(1)\to 0$ as $n\to\infty$. Taking the liminf as $n\to\infty$ we get
$$
\lambda(A) \le \liminf_{n\to\infty} \int_A \dist((x, \tilde u_n, \tilde r_n), \D)\dd x \le \lambda(C) + \varepsilon
$$
and since $\varepsilon$ is arbitrary and $C \subset\subset A$, we conclude that $\lambda(A) \le \sup_{B \subset\subset A} \lambda(B)$ as claimed.

\medskip

By applying the de Giorgi-Letta result of Theorem~\ref{thm_degiorgi} we know that $A\mapsto \scm I_A$ is a regular measure on the Borel $\sigma$-algebra of $\Omega$.  To see that it is also a Radon measure recall from Lemma~\ref{prop:fn:f}(i) that
$$
f_u(x,r(x)) \le c\left(1+|x|^2+|u(x)|^2+|r(x)|^2\right), \quad x\in \Omega
$$
so that $A\mapsto \scm I_A(u, r)$ is finite on compact sets. This allows us to finish the proof by applying the Radon-Nikodym Theorem, \cite[Theorem 6.10]{rudin}.
\end{proof}

\subsection{Identification of $\frac{d \scm I}{d\mathcal L^N}$}\label{sec_id:2}

From Theorem~\ref{abs:cont} we have the existence of a unique function $F\in L^1(\Omega)$ such that $\frac{d \scm I}{d\mathcal L^N}=F$.  The final goal for the proof of Theorem~\ref{thm:main_Istar_scm_I} is now to show that 
$$
F(x) = \conv_r \dist ( (x,u(x), r(x)), \D )
$$
where the lower convex envelope w.r.t. $r$ is defined in Definition~\ref{def:conv}. We will reach our goal with Lemma~\ref{dIQ} below by showing that $F$ coincides with the quasiconvexification $\mathcal{Q}_r \dist(\left(x_0, u\left(x_0\right), r\left(x_0\right)\right), \D)$ w.r.t. the variable $r$. The result then follows from the identification $\conv_r \dist ( (x,u(x), r(x), \D ) = \mathcal{Q}_r \dist(\left(x_0, u\left(x_0\right), r\left(x_0\right)\right), \D)$ as shown in Lemma~\ref{equality}.

\medskip

We begin with the definition of the quasiconvexification with respect to $r$. { First of all, we recall that the quasiconvexity concept was first introduced by Morrey \cite{Morrey} while the quasiconvexification of a measurable function was given by Dacorogna \cite[Theorem 6.9]{dacorogna}.} Let 
$$
C_{1-\mathrm{per}}^{\infty}(\mathbb{R}^N)=\{w\in C^{\infty}(\mathbb{R}^N) : w(y+e_i)=w(y) \mbox{ for all } i=1,\dots ,N \mbox{ and all } y\in \R^N\}
$$
where $(e_1, \dots ,e_N)$ is the standard basis of $\R^N$. Recall also $Q \coloneqq { (-1/2,1/2)}^N$.

\begin{definition} \label{def:quasiconv} Let $f: Y \to [0,\infty]$ be a Carath\'{e}odory-function and fix $(x,u)\in \Omega\times [0,\infty)$. We define the quasiconvexification $\mathcal{Q}f(x,u,\cdot)$ with respect to $r$ of the map $r \mapsto f(x,u,r)$ as 
$$
\mathcal{Q}_r f(x, u, r)\coloneqq \inf_{w} \left\{\int_Q f(x, u, r+w(y)) \dd y: w \in C_{1-\mathrm{per}}^{\infty}(\mathbb{R}^N), \int_Q w(y) \dd y=0\right\}
$$
for all $r\in \R$.
\end{definition}
\begin{remark} \label{rem:L1_per}
    In Definition~\ref{def:quasiconv} one can replace $C_{1-\mathrm{per}}^{\infty}(\mathbb{R}^N)$ by $L^\infty_{1-\mathrm{per}}(\R^N)$, where
    $$
L_{1-\mathrm{per}}^{\infty}(\mathbb{R}^N)=\{w\in L^{\infty}(\mathbb{R}^N) : w(y+e_i)=w(y) \mbox{ for all } i=1,\dots ,N \mbox{ and all } y\in \R^N\}.
$$
     Indeed, take $\varphi \in C^\infty_c(\R^N)$ such that $\operatorname{supp}(\varphi) \subset B_1(0)$, $\varphi \geq 0$ and $\int_{\R^N} \varphi(z)\dd z = 1$ and let $\varphi_\varepsilon(z) \coloneqq \frac{1}{\varepsilon^N} \varphi\left(\frac{z}{\varepsilon}\right)$ for all $\varepsilon>0$. For any $w\in L_{1-\mathrm{per}}^{\infty}(\mathbb{R}^N)$ such that $\int_Q w(y) \dd y=0$, we can define $w_\varepsilon\in C_{1-\mathrm{per}}^{\infty}(\mathbb{R}^N)$ as $w_\varepsilon(y) \coloneqq (w * \varphi_\varepsilon)(y)$. Then also $\int_Q w_\varepsilon(y) \dd y=0$ and, up to subsequences, 
    $$\lim_{\varepsilon \to 0} \int_Q f(x, u, r+w_{\varepsilon}(y)) \dd y = \int_Q f(x, u, r+w(y)) \dd y,$$
    by the Dominated Convergence Theorem since $w_\varepsilon \to w$ almost everywhere as $\varepsilon \to 0$ and the sequence is uniformly bounded.
\end{remark}

The next result, which shows that the quasiconvexification of the distance function $\dist(\cdot,\mathcal{D})$ w.r.t the last variable $r$ coincides with its lower convex envelope w.r.t. $r$, is maybe well known. For the convenience of the reader, and since we could not locate a reference to the result in the case where the perturbation function $w$ is periodic, we also give the proof. 

\begin{lemma}\label{equality}
$$\mathcal{Q}_r \dist(\left(x, u\left(x\right), r\left(x\right)\right), \D)=\conv_r \dist ( (x,u(x), r(x)), \D ).$$
\end{lemma}

\begin{proof}
First, pick any test function $w \in C^\infty_{1-\text{per}}(\mathbb{R}^N)$ with $\int_Q w(y) \dd y= 0$  and a convex function $g: \mathbb{R} \to \mathbb{R}$  such that $g \le f_{x,u}$, where $f_{x,u}$ is defined in \eqref{def:dist:fix}. By Jensen's inequality
$$
\int_Q f_{x,u}(r + w(y)) \dd y\ge \int_Q g(r + w(y)) \dd y\ge g\left( \int_Q (r + w(y)) \dd y\right)= g(r),
$$
because $r$ is constant in $y$ and $w$ has zero mean.
Taking the supremum over $g$ we get $\int_Q f_{x,u}(r + w(y)) \dd y\ge \conv_r f_{x,u}(r)$ and taking the infimum over $w$, we conclude that $\mathcal{Q}_r f_{x,u}(r) \ge \conv_r f_{x,u}(r)$.

To get equality, we apply Carath\'{e}odory's Theorem \cite[Theorem 1.3]{Eck} (see also \cite[Theorem 17.1]{Rockafellar}) to $(r, \conv_r f_{x,u}(r))$, which belongs in the convex hull of $(r, f_{x,u}(r)),\, r\in \R$, so that there exist $\lambda_i$ and $r_i$, $i=1,2, 3$ such that $\sum_{i=1}^3 \lambda_i = 1$ and $(r,\conv_r f_{x,u}(r)) = \sum_{i=1}^3 \lambda_i(r_i, f_{x,u}(r_i))$. We now divide $Q = (-1/2, 1/2)^N$ into three disjoint sets  $E_1, E_2, E_3$ such that
$|E_i| = \lambda_i$ for $i=1,2,3$. Define $w_0: Q \to \R$ as
$$
w_0(y) =
\begin{cases}
    r_1 - r & \text{if } y \in E_1 \\  
    r_2 - r & \text{if } y \in E_2  \\
     r_3 - r & \text{if } y \in E_3
\end{cases}
$$
and we extend it to all of $\R^N$ by $1$-periodicity. Clearly, the main value of $w_0$ over $Q$ vanishes since
$$
\int_Q w_0(y) \dd y = \lambda_1(r_1 - r) + \lambda_2(r_2 - r)+\lambda_3(r_3 - r) = 0.
$$
Evaluation of the integral of $f_{x,u}$ yields
\begin{align*}
\int_Q f_{x,u}(r + w_0(y)) \dd y &= \int_{E_1} f_{x,u}(r_1)\dd y + \int_{E_2} f_{x,u}(r_2)\dd y + \int_{E_3} f_{x,u}(r_3)\dd y\\ &= \lambda_1 f_{x,u}(r_1) + \lambda_2 f_{x,u}(r_2) + \lambda_3 f_{x,u}(r_3) = \conv_r f_{x,u}(r). 
\end{align*}
Although $w_0$ is not smooth, it is bounded and periodic, which is sufficient to  conclude that $\mathcal{Q}_r f_{x,u}(r) \cancel{\le} = \conv_r f_{x,u}(r)$, cf. Remark~\ref{rem:L1_per}.
\end{proof}

\begin{lemma}\label{dIQ}
For a.e. $x_0 \in \Omega$ we have $F(x_0)=\mathcal{Q}_r \dist(\left(x_0, u\left(x_0\right), r\left(x_0\right)\right), \D)$. 
\end{lemma}

\begin{proof} 
Since $F=\frac{d \scm I}{d \mathcal L^N}$ we have for almost all $x_0\in \Omega$
\begin{equation}\label{dsc:ln}
F(x_0)= \lim _{\delta \to 0^{+}} \frac{\scm I_{Q_\delta(x_0)}(u,r)}{|Q_\delta(x_0)|}<\infty
\end{equation}
where we recall $Q_\delta(x_0)\coloneqq x_0 + \delta Q$.
By the Lebesgue differentiation Theorem we can restrict to Lebesgue points $x_0\in \Omega$ of $u$ and $r$, i.e., points $x_0$ where 
\begin{equation}\label{leb:point}
\lim _{\delta \to 0^{+}} \frac{1}{\delta^N} \int_{Q_\delta(x_0)}\left|u(x)-u\left(x_0\right)\right|^2 \dd x=\lim _{\delta \to 0^{+}} \frac{1}{\delta^N} \int_{Q_\delta(x_0)}\left|r(x)-r\left(x_0\right)\right| \dd x=0.
\end{equation}
Note that \eqref{leb:point} also implies $\lim _{\delta \to 0^{+}} \frac{1}{\delta^N} \int_{Q_\delta(x_0)}\left|r(x)-r\left(x_0\right)\right|^2 \dd x=0 $ since $r\in L^\infty(\Omega)$.
By Lemma~\ref{scm_alt_def} and for $\delta>0$ fixed, let $(r_{n,\delta}) \subset L^\infty(\Omega)$ be such that $r_{n,\delta} \stackrel{*}{\rightharpoonup} r$ in $L^\infty(\Omega)$ as $n\to\infty$ and
$$
\lim _{n\to\infty} \int_{Q_\delta(x_0)} f_u\left(x, r_{n,\delta}(x)\right) \dd x \le \scm I_{Q_\delta(x_0)}(u,r)+\delta^{N+1},
$$
where $f_u(x,s)$ is defined in \eqref{def:dist:fix}. Then, by \eqref{dsc:ln}, we get
$$
\begin{aligned}
F(x_0) &\ge \liminf_{\delta \to 0^{+}} \lim _{n\to\infty} \frac{1}{\delta^N} \int_{Q_\delta(x_0)} f_u\left(x, r_{n,\delta}(x)\right) \dd x\\&=\liminf_{\delta \to 0^{+}} \lim _{n\to\infty} \int_{Q} f_u\left(x_0+\delta y, r\left(x_0\right)+w_{n,\delta}(y)\right) \dd y
\end{aligned}
$$
by the change of variable $x\mapsto x_0 +\delta y$, where $\hat w_{n,\delta}(y)\coloneqq r_{n,\delta}\left(x_0+\delta y\right)-r\left(x_0\right)$. We claim that $\hat w_{n,\delta} \stackrel{*}{\rightharpoonup} 0$ in $L^\infty\left(Q\right)$ if we first let $n\to\infty$ and then $\delta \to 0^{+}$. Indeed let $\varphi \in L^{1}(Q)\cap L^\infty(Q)$. Using a change of variables, we get
$$
\begin{aligned}
\left|\int_Q \varphi(y) \hat w_{n,\delta}(y) \dd y\right| \le & \left|\int_Q \varphi(y)\left(r_{n,\delta}\left(x_0+\delta y\right)-r\left(x_0+\delta y\right)\right) \dd y\right| \\
& +\left|\int_Q \varphi(y)\left(r\left(x_0+\delta y\right)-r\left(x_0\right)\right) \dd y\right| \\
\le & \left|\frac{1}{\delta^N} \int_{Q_\delta(x_0)} \varphi\left(\left(\frac{x-x_0}{\delta}\right)\right)\left(r_{n,\delta}(x)-r(x)\right) \dd x\right| \\
& +\|\varphi\|_{L^{\infty}(Q)}\left(\frac{1}{\delta^N} \int_{Q_\delta(x_0)}\left|r(x)-r\left(x_0\right)\right| \dd x\right).
\end{aligned}
$$
If we now let $n\to\infty$ the first integral tends to zero, since $r_{n, \delta} \stackrel{*}{\rightharpoonup} r$ in $L^\infty\left(Q_\delta(x_0)\right)$. The second integral tends to $0$ letting $\delta \to 0^{+}$ and by using \eqref{leb:point}. Since the functions $\hat w_{n,\delta}$ are uniformly bounded in $L^\infty(Q)$ and since $L^1(Q)\cap L^\infty(Q)$ is dense in $L^1(Q)$, we have indeed established that $\hat w_{n, \delta} \stackrel{*}{\rightharpoonup} 0$ in $L^\infty\left(Q\right)$ if we first let $n\to\infty$ and then $\delta \to 0^{+}$.

\medskip

Using a diagonalization procedure we get a sequence $\hat{w}_n \in L^\infty\left(Q\right)$ such that $\hat{w}_n \stackrel{*}{\rightharpoonup} 0$ in $L^\infty\left(Q\right)$ and
\begin{equation}\label{dsc:1}
F(x_0) \ge \liminf _{n\to\infty} \int_Q f_u\left(x_0+\delta_n y, r\left(x_0\right)+\hat{w}_n(y)\right) \dd y,
\end{equation}
where $\delta_n \to 0$. 
Let $w_n \coloneqq \hat w_n - \int_Q \hat w_n(x)\dd x$ so that $w_n-\hat w_n\to 0$ in $L^\infty(Q)$.  

Now we apply Lemma~\ref{lem:I_compare} by choosing $u' = u$ and comparing the function $f_u$ at two distinct configurations in $\Omega \times \mathbb{R}$, i.e., $(x, r) = (x_0 + \delta_n y, r(x_0) + \hat{w}_n(y))$ and $(x', r') = (x_0, r(x_0) + w_n(y))$. 
With this choice, the metric variation term simplifies to $\interleave \delta_n y, u(x_0 + \delta_n y) - u(x_0), \hat{w}_n(y) - w_n(y) \interleave$. Then, $\int_Q |f_u(x_0+\delta_ny, r(x_0)+\hat w_n(y)) - f_u(x_0, r(x_0)+w_n(y))|\dd y\to 0$ as $n\to\infty$ since $x_0$ is a Lebesgue point of $u$, $w_n-\hat w_n\to 0$ in $L^\infty(Q)$ and $\delta_n\to 0$ as $n\to \infty$. 
Thus, also from \eqref{dsc:1}, we get 
$$
\begin{aligned}
F(x_0) &\ge \liminf_{n\to\infty} \int_Q f_u\left(x_0+\delta_n y, r\left(x_0\right)+\hat{w}_n(y)\right) \dd y \\
& = \liminf_{n\to\infty} \int_Q f_u\left(x_0, r\left(x_0\right)+w_n(y)\right) \dd y\\
& = \liminf_{n\to\infty} \int_Q \dist \left( \left(x_0, u\left(x_0\right), r(x_0)+w_n(y)\right), \D\right) \dd y
\\ &
\ge \mathcal{Q}_r \dist \left(\left(x_0, u(x_0), r(x_0)\right),\D\right)
\end{aligned}
$$
where the last inequality follows from the definition of the quasiconvexification $\mathcal{Q}_r$ in Definition~\ref{def:quasiconv}. 

\medskip

To conclude the proof it remains to show that
\begin{equation*}
F(x_0)\le \mathcal{Q}_r  \dist \left( \left(x_0, u(x_0), r(x_0)\right), \D\right) \quad\text { for a.e. } x_0 \in \Omega.
\end{equation*}
For a given Lebesgue point $x_0\in \Omega$ satisfying \eqref{leb:point} fix $\eta>0$ and let $w \in C_{1\text{-per}}^{\infty}\left(\mathbb{R}^N\right)$ be such that $\int_Q w(y) \dd y=0$ and
\begin{equation}\label{start:2way}
\int_Q f_u\left(x_0, r\left(x_0\right)+w(y)\right) \dd y \leq \mathcal{Q}_r  \dist \left( \left(x_0, u(x_0), r(x_0)\right) ,\D\right)+\eta.
\end{equation}
For any $\delta>0$ and $n\in\N$ set $w_{n, \delta}(x)\coloneqq w\left(n\left(x-x_0\right) / \delta\right)$. Then $w_{n, \delta} \stackrel{*}{\rightharpoonup} 0$
in $L^\infty\left(Q_\delta(x_0)\right)$ as $n\to\infty$, see Lemma \ref{lemma_limit} below. Hence, by Lemma \ref{scm_alt_def},
\begin{align*} 
F(x_0)=& \lim _{\delta \to 0^{+}} \frac{\scm I_{Q_\delta(x_0)}(u,r)}{\delta^N}  \\
\le & \liminf_{\delta \to 0^{+}} \liminf_{n\to\infty} \frac{1}{\delta^N} \int_{Q_\delta(x_0)} f_u\left(x, r(x)+w_{n, \delta}(x)\right) \dd x\\
= & \liminf_{\delta \to 0^{+}}  \liminf_{n\to\infty} \int_Q f_u\left(x_0 + \delta y, r(x_0 + \delta y)+w(ny)\right) \dd y \\
= & \liminf_{\delta \to 0^{+}}  \liminf_{n\to\infty} \int_Q f_u\left(x_0, r(x_0)+w(ny)\right) \dd y,
\end{align*}
where in the last step we have used Lemma~\ref{lem:I_compare}
similarly as before, by evaluating $f_u$ at $(x_0 + \delta y, r(x_0 + \delta y)+w(ny))$ and at $(x_0, r(x_0)+w(ny))$. In particular, in this case, the bounding metric term reduces to $\interleave \delta y, u(x_0 + \delta y) - u(x_0), r(x_0 + \delta y) - r(x_0) \interleave$, which is notably independent of $n$. As $\delta \to 0^+$, this residual distance vanishes since $x_0$ is a Lebesgue point of $u$ and $r$, so that the difference between the integrands converges to zero.
Finally, note that the $1$-periodicity of the function $z\mapsto w(z)$ in all coordinate directions implies that 
$$
\int_Q f_u\left(x_0, r(x_0)+w(ny)\right) \dd y = \frac{1}{n^N}\int_{nQ} f_u\left(x_0, r(x_0)+w(z)\right) \dd z = \int_Q f_u\left(x_0, r(x_0)+w(z)\right) \dd z.
$$
Therefore, we can continue with the previous estimate and get 
$$
F(x_0) \leq \liminf_{\delta \to 0^{+}} \int_Q f_u\left(x_0, r(x_0)+w(z)\right) \dd z \leq \mathcal{Q}_r  \dist\left(x_0, u(x_0), r(x_0)\right)+\eta
$$
by \eqref{start:2way}. It now suffices to let $\eta \to 0^{+}$.
\end{proof}

The final result of this section is a technical result that has been used in the proof of the previous Lemma~\ref{dIQ}. It is based on a Fourier-series argument.

\begin{lemma}\label{lemma_limit}
Let $x_0\in \Omega$, $\delta>0$ and $w \in C_{1-\mathrm{per}}^{\infty}\left(\mathbb{R}^N\right)$ such that $\int_Q w(x)\dd x=0$. Then 
$$
\lim_{n\to\infty} \int_{Q_\delta(x_0)} w_{n,\delta}(x)\varphi(x)\dd x =0
$$
for every $\varphi\in L^1(Q_\delta(x_0))$, where $w_{n,\delta}(x)\coloneqq w(n(x-x_0)/\delta)$.
\end{lemma}
\begin{proof}
    For $\psi(x) \coloneqq \varphi(x_0+\delta x)$ we have $\psi\in L^1(Q)$ and 
    $$
    \int_{Q_\delta(x_0)} w_{n,\delta}(x)\varphi(x)\dd x = \delta^N \int_Q w(nx)\psi(x)\dd x.
    $$
    Let $e_k(x) \coloneqq e^{2\pi i k\cdot x}$, $k\in\Z^N$ be the standard Fourier basis of $L^2(Q)$
    so that $w(x)=\sum_{k\in \Z^N} \hat w_ke_k(x)$ and $\psi(x)=\sum_{k\in\Z^N} \hat\psi_k e_k(x)$. Since $w$ is a $1$-periodic $C^{2N}(\R^N)$-function , we have $|\hat w_k|\leq \frac{c_1}{1+|k|^{2N}}$ while for $\psi\in L^1(Q)$ we have $|\hat \psi_k|\leq c_2$ for some constants $c_1, c_2>0$ and $\lim_{|k|\to\infty} \hat\psi_k=0$ by the Riemann-Lebesgue lemma. Moreover,  by the dominated convergence theorem we see that 
    $$
    \lim_{n\to \infty} \int_Q w(nx)\psi(x)\dd x = \lim_{n\to \infty} \sum_{k\in \Z^N} \hat w_k \hat \psi_{-nk}=0
    $$
    since $n|k|\to\infty$ for $k\not =0$ and $\hat w_0=0$ by assumption.
\end{proof}

\section{Appendix}\label{appendix}

In \eqref{def:sc-} we have defined the lower semi-continuous relaxation $\scm I$ of $I$.  In Lemma~\ref{lemma_sc-} below we show that $\scm I$ coincides with 
\begin{equation} \label{def:hatsc-}
\hscm I(u,r)\coloneqq\sup \{G(u,r) : G \mbox{ is weak-weak$^*$ lower semi-continuous, } G\le I \mbox{ on } V\}
\end{equation} 
and, as a consequence, that it is weak-weak$^\ast$ lower semi-continuous on $V$. We finish the appendix with the proof of Lemma~\ref{scm_alt_def} and with a frequently used technical lemma.

Before starting with a second characterization of $\scm I$, we will prove that the infimum in the definition \eqref{def:sc-} of $\scm I(u,r)$ is attained.

\begin{lemma}\label{sc:min}
For any $(u, r) \in H^1(\Omega) \times L^\infty(\Omega)$
$$
\scm I(u,r) = \min\left\{ \liminf_{n\to\infty} I(u_n,r_n) : u_n \rightharpoonup u \mbox{ in } H^1(\Omega) \mbox{ and } r_n\stackrel{*}{\rightharpoonup} r \mbox{ in } L^\infty(\Omega)\right\}.
$$
\end{lemma}

\begin{proof} Fix $(u,r)\in V$ and let 
$$
L\coloneqq \inf\left\{ \liminf_{n\to\infty} I(u_n,r_n) : u_n \rightharpoonup u \mbox{ in } H^1(\Omega) \mbox{ and } r_n\stackrel{*}{\rightharpoonup} r \mbox{ in } L^\infty(\Omega)\right\}.
$$
For each $k \in \mathbb{N}$, by the properties of the infimum, there exists a sequence $(u_n^k, r_n^k)_{n}$ such that
\begin{align}\label{ukn:seq}
    \text{$u_n^k \rightharpoonup u$ in $H^1(\Omega)$ and $r_n^k \stackrel{*}{\rightharpoonup} r$ in $L^\infty(\Omega)$ as $n\to\infty$}
    \end{align}
    \begin{align}\label{ukn:I:seq}
    \text{$\displaystyle \lim_{n\to\infty} I(u_n^k, r_n^k) < L + \frac{1}{k}$.}
\end{align}
Since for each $k$ the sequence $(u_n^k, r_n^k)$ converges weak-weak$^\ast$ to $(u, r)$ as $n\to \infty$ we can use \eqref{ukn:seq}-\eqref{ukn:I:seq} to select a sufficiently large index $n(k)$ such that
\begin{align}\label{conv:dw}
d_w(u_{n(k)}^k, u) < \frac{1}{k}, \quad d_{w*}(r_{n(k)}^k, r) < \frac{1}{k}
\end{align}
and 
\begin{align}\label{decay:Ik}
 I(u_{n(k)}^k, r_{n(k)}^k) < \displaystyle \lim_{n\to\infty} I(u_n^k, r_n^k) + \frac{1}{k} < L + \frac{2}{k}
 \end{align}
hold, where $d_w$, $d_{w*}$ are the metrics which locally describe the weak topology in $H^1(\Omega)$ and the weak$^\ast$ topology in $L^\infty(\Omega)$, respectively. Define the diagonal sequence $(u_k, r_k) \coloneqq (u_{n(k)}^k, r_{n(k)}^k)$. By \eqref{conv:dw} and \eqref{decay:Ik}, we get
    $$
    u_k\rightharpoonup u, \quad r_k  \stackrel{*}{\rightharpoonup} r \quad \text{as} \quad {k \to \infty}, \quad  \limsup_{k \to \infty} I(u_k, r_k) \le L.
    $$
Furthermore, $\liminf_{k \to \infty} I(u_k, r_k) \ge L$ by the definition of $\scm I$. Therefore $\lim_{k \to \infty} I(u_k, r_k) = L$ which shows that $L$ is attained.
\end{proof}

\begin{lemma}\label{lemma_sc-} The functional $\scm I$ is weak-weak$^\ast$ lower semi-continuous and the two defintions \eqref{def:sc-} and \eqref{def:hatsc-} coincide, i.e.,
$$
\scm I=\hscm I.
$$
\end{lemma}

\begin{remark} If $I$ is coercive, then $\scm I$ is coercive and has at least one minimizer. Moreover, the infima of $\scm I$ and $I$ coincide, i.e., $\min_{(u,r)\in V} \scm I(u,r)=\inf _{(u,r)\in V} I(u,r)$. Finally, if $(u,r)$ is the limit of a minimizing sequence for $I$, then $(u,r)$ is a minimizer for $\scm I$ and vice versa. These statements can be found in \cite[Pag. 30]{DalMaso}.
\end{remark}

\begin{proof}
First note that $\scm I(u,r) \le I(u,r)$ for all $(u,r)\in\mathcal{C}$, just considering the constant sequence $u_n\equiv u$ and $r_n\equiv r$. 

We now prove that $\scm I$ is weak-weak$^\ast$ lower semicontinuous. 
Note that, up to subsequence
$$
\liminf_{n\to\infty} \operatorname{sc}^- I(u_n, r_n)=\lim_{n\to\infty} \operatorname{sc}^- I(u_n, r_n).
$$
Moreover, by Lemma~\ref{sc:min} applied to $\scm I(u_n, r_n)$, for each $n \in \mathbb{N}$, there exists a sequence $(u_{n,j}, r_{n,j})_{j \in \mathbb{N}}$ such that $u_{n,j}\rightharpoonup u_n$ in $H^1(\Omega)$, $r_{n,j}\stackrel{*}{\rightharpoonup} r_n$ in $L^\infty(\Omega)$ as $j\to\infty$ and
 $$
 \lim_{j \to \infty} I(u_{n,j}, r_{n,j}) = \operatorname{sc}^- I(u_n, r_n).
 $$
Using the local metric $d$ for the weak-weak$^\ast$ topology on $V$, for each $n$ we can choose a sufficiently large index $j_n$ such that:
\begin{equation}\label{ukjk}
d\big((u_{n,j_n}, r_{n,j_n}), (u_n, r_n)\big) < \frac{1}{n} \quad \text{and} \quad |I(u_{n,j_n}, r_{n,j_n}) - \operatorname{sc}^- I(u_n, r_n)| < \frac{1}{n}.
\end{equation}
By the triangle inequality:
$$
d\big((u_{n,j_n}, r_{n,j_n}), (u, r)\big) \le d\big((u_{n,j_n}, r_{n,j_n}), (u_n, r_n)\big) + d\big((u_n, r_n), (u, r)\big) \to 0 \quad \text{as } n\to\infty.
$$
Thus, $u_{n,j_n}\rightharpoonup u$ in $H^1(\Omega)$ and $r_{n,j_n}\stackrel{*}{\rightharpoonup} r$ in $L^\infty(\Omega)$. By the definition of $\scm I(u,r)$, we have
$$
\scm I(u,r) \le \liminf_{n\to\infty} I(u_{n,j_n}, r_{n,j_n})=\liminf_{n\to\infty} \scm I(u_n, r_n)
$$
by \eqref{ukjk}, and hence $\scm I$ is weak-weak$^\ast$ lower semi-continuous.

\medskip

It remains to show that $\scm I = \hscm I$. It follows form the definition of $\hscm I(u,r)$  and the weak-weak$^\ast$ lower semi-continuity of $\scm I$ that $\hscm I(u,r)\ge \scm I(u,r)$. On the other hand, take a weak-weak$^\ast$ lower semicontinuous function $G$ such that $G\le I$. Given a sequence $(u_n,r_n)$ which weak-weak$^\ast$ converges to $(u,r)$ we have
$$
G(u,r)\le \liminf_{n\to\infty} G(u_n,r_n)\le \liminf_{n\to\infty} I(u_n,r_n).
$$
Since it is valid for all sequences $(u_n, r_n)$, we can take the infimum over $(u_n, r_n)$, giving that
$$
G(u,r)\le \scm I(u,r).
$$
Since the last inequality is valid for all weak-weak$^\ast$ lower semicontinuous functions $G$, we can take the supremum and obtain the reverse inequality $\hscm I \leq \scm I$.
\end{proof}

In what follows we provide the proof of Lemma~\ref{scm_alt_def} which is a key ingredient in Section \ref{sec_id:1} as a first step to prove Theorem~\ref{thm:main_Istar_scm_I}.

\begin{proof}[Proof of Lemma~\ref{scm_alt_def}:] 
For $(u,r)\in \cC$ let
\begin{equation} \label{def:h}
h(u,r) \coloneqq \inf \left\{\liminf_{n\to\infty} \int_\Omega f_u(x,r_n(x))\dd x: (r_n)  \subset L^\infty(\Omega), r_n\stackrel{*}{\rightharpoonup} r \text { in } L^{\infty}(\Omega) \right\},
\end{equation}
where $f_u$ is defined in \eqref{def:dist:fix}, and furthermore we set $h=\infty$ on $V\setminus\cC$. 
The proof consists of showing $\scm I\leq h$ and $\scm I\geq h$.

\medskip

Since obviously $I \geq h$ we can show $\scm I\geq h$ by checking that $h: V \to \R$ is weak-weak$^\ast$ lower semi-continuous and applying Lemma~\ref{lemma_sc-}. Since $\cC$ is weak-weak$^\ast$ closed by Lemma~\ref{lemma_closed}, it is enough to check this on $\cC$. First, in analogy to Lemma~\ref{sc:min}, one can show that the $\inf$ in \eqref{def:h} is a min. Next, take a sequence $(u_n,r_n)$ in $\cC$ and a point $(u,r)$ in $\cC$ such that $u_n\rightharpoonup u$ in $H^1(\Omega)$, $r_n\stackrel{\ast}{\rightharpoonup} r$ in $L^\infty(\Omega)$. For checking the weak-weak$^\ast$ lower semi-continuity we can assume that $\liminf_{n\to\infty} h(u_n,r_n)<\infty$. Let us take a subsequence $(\tilde u_n,\tilde r_n)$ of $(u_n, r_n)$ such that 
$$
\liminf_{n\to\infty} h(u_n,r_n) = \lim_{n\to\infty} h(\tilde u_n,\tilde r_n).
$$
If we use the min-variant of \eqref{def:h} then there exists a sequence $(\tilde r_{n,j})$ in $L^\infty(\Omega)$ such that $\tilde r_{n,j}\stackrel{\ast}{\rightharpoonup} \tilde r_n$ as $j\to\infty$ and 
$$
h(\tilde u_n,\tilde r_n) = \lim_{j\to \infty} \int_\Omega f_{\tilde u_n}(x,\tilde r_{n,j}(x))\dd x.
$$
For each $n\in\N$ we may choose the index $j_n$ so large that 
$$
d\left( (\tilde u_n, \tilde r_{n, j_n}), (\tilde u_n, \tilde r_n)\right)< \frac{1}{n} \mbox{ and } \left|\int_\Omega f_{\tilde u_n}(x,\tilde r_{n, j_n}(x))\dd x- h(\tilde u_n, \tilde r_n)\right|<\frac{1}{n},
$$
recalling that $d$ is the local metric for the weak-weak$^\ast$ topology on $V$. The triangle inequality 
$$
d\big((\tilde u_{n}, \tilde r_{n,j_n}), (u, r)\big) \le d\big((\tilde u_{n}, \tilde r_{n,j_n}), (\tilde u_n, \tilde r_n)\big) + d\big((\tilde u_n, \tilde r_n), (u, r)\big) \to 0 \quad \text{as } n\to\infty
$$
shows that $\tilde u_n\rightharpoonup u$ in $H^1(\Omega)$, $\tilde r_{n,j_n}\stackrel{\ast}{\rightharpoonup} \tilde r$ in $L^\infty(\Omega)$ as $n\to \infty$. For the next step note first that the assumption $\liminf_{n\to\infty} h(u_n,r_n)<\infty$ implies that also $\liminf_{n\to\infty} \int_\Omega f_{\tilde u_n}(x,\tilde r_{n,j_n}(x))\dd x<\infty$. Then the definition of $h$ together with Lemma~\ref{lem:I_compare} yield
\begin{align*}
h(u,r) \leq & \liminf_{n\to\infty} \int_\Omega f_u(x,\tilde r_{n,j_n}(x))\dd x \\
\leq & \liminf_{n\to\infty} \int_\Omega f_{\tilde u_n}(x,\tilde r_{n,j_n}(x))\dd x + 2\lim_{n\to\infty} \left(\sqrt{\int_\Omega f_{\tilde u_n}(x,\tilde r_{n,j_n}(x))\dd x}\|\tilde u_n-u\|_2+ \|\tilde u_n-u\|_2^2\right) \\
=& \liminf_{n\to\infty}  h(\tilde u_n,\tilde r_n) \\
=& \liminf_{n\to\infty} h(u_n,r_n)
\end{align*}
where we have used that $ \tilde u_n \to u$ in $L^2(\Omega)$ as $n\to\infty$.

\medskip

It remains to show $\scm I \leq h$ on $\cC$. So let us consider $(u,r)\in \cC$ with $u\not =0$. From the strong minimum principle,  see \cite[Chapter 6]{PS07}, we get that $\inf_\Omega u>0$. Using the min-characterization in the definition \eqref{def:h} of $h$ let us choose a sequence $r_n\stackrel{\ast}{\rightharpoonup}r$ in $L^\infty(\Omega)$ such that $h(u,r)=\lim_{n\to\infty} \int_\Omega f_u(x,r_n(x))\dd x$. Next we choose an arbitrary $\alpha>0$ and define $\tilde u_n\in H^1(\Omega)$ as the unique weak solution of the boundary value problem
$$
-\Delta \tilde u_n + \alpha\tilde u_n = (\alpha+r_n)u \mbox{ in } \Omega, \quad \partial_\nu \tilde u_n=0 \mbox{ on } \partial\Omega.
$$
First we show that $\tilde u_n\to u$ in $H^1(\Omega)$ and in $L^\infty(\Omega)$ as $n\to \infty$. To see this, let us define $w_n\coloneqq\tilde u_n-u\in H^1(\Omega)$ and observe that $w_n$ satisfies
\begin{equation}\label{eq:wn}
-\Delta w_n + \alpha w_n = (r_n-r)u \mbox{ in } \Omega, \quad \partial_\nu w_n=0 \mbox{ on } \partial\Omega.
\end{equation}
Passing to the weak formulation of \eqref{eq:wn} with any $v\in H^1(\Omega)$ as test function, we get
\begin{equation}\label{weak:eq:wn}
\int_\Omega \nabla w_n \nabla v \dd x+\alpha\int_\Omega w_n v \dd x=\int_\Omega (r_n-r)u v \dd x.
\end{equation}
By choosing $v=w_n$ and using H\"older's inequality, we get
\begin{equation}\label{weak:eq:wn:est}
\int_\Omega |\nabla w_n|^2 \dd x+\alpha\int_\Omega |w_n|^2 \dd x=\int_\Omega (r_n-r)u w_n \dd x\le \|r_n-r\|_\infty \|u\|_2 \|w_n\|_2 \leq C\|w_n\|_2
\end{equation}
since $(r_n)$ is bounded in $L^\infty(\Omega)$. As the left hand side of the previous formula is equivalent to $\|w_n\|_{H^1}^2$, we have that $(w_n)$ is bounded in $H^1(\Omega)$, which implies the existence of $w\in H^1(\Omega)$ such that $w_n\rightharpoonup w$ in $H^1(\Omega)$ and $w_n\to w$ in $L^2(\Omega)$ by the Rellich Kondrachov Theorem and where we have used that $\partial\Omega$ is Lipschitz. 
From \eqref{weak:eq:wn} together with $r_n\stackrel{\ast}{\rightharpoonup}r$ in $L^\infty(\Omega)$ and $w_n\rightharpoonup w$ in $H^1(\Omega)$ we get
$$\int_\Omega \nabla w \nabla v \dd x+\alpha\int_\Omega w v \dd x=0$$
for all $v\in H^1(\Omega)$. In particular, by choosing $v=w$, we conclude that $w=0$, so that $w_n\rightharpoonup 0$ in $H^1(\Omega)$ and $w_n\to 0$ in $L^2(\Omega)$.  Returning to \eqref{weak:eq:wn:est} we find that $w_n\to 0$ in $H^1(\Omega)$ which gives us $\tilde u_n\to u$ in $H^1(\Omega)$. \cite[Theorem 8.22]{GT} Now we use the De Giorgi-Nash theory \cite[Theorem 8.22]{GT} which implies that $(w_n)$ is uniformly bounded in $C^{0,\beta}(\overline{\Omega})$ for some $\beta\in (0,1)$ and hence $\tilde u_n\to u$ uniformly on $\Omega$ by Ascoli Arzel\'a Theorem.

Next, we define 
$$
\tilde r_n \coloneqq \frac{-\Delta \tilde u_n}{\tilde u_n}=(\alpha+r_n)\frac{u}{\tilde u_n}-\alpha
$$
and find that 
$$
\tilde r_n-r_n = (\alpha+r_n)\left(\frac{u}{\tilde u_n}-1\right)\to 0 \mbox{ uniformly on } \Omega.
$$
Therefore, $(\tilde u_n, \tilde r_n)\in \cC$, $\tilde u_n\to u$ in $H^1(\Omega)$, $\tilde r_n\stackrel{\ast}{\rightharpoonup}r$ in $L^\infty(\Omega)$ as $n\to \infty$ which implies 
\begin{align*}
\scm I(u,r) & \leq \liminf_{n\to\infty} \int_\Omega f_{\tilde u_n}(x,\tilde r_n(x))\dd x \\
& \leq  \liminf_{n\to\infty} \int_\Omega f_u(x,\tilde r_n(x))\dd x \quad \mbox{ (where we used $\tilde u_n-u\to 0$ in $L^2(\Omega)$ and Lemma~\ref{lem:I_compare}) }\\
& = \liminf_{n\to\infty} \int_\Omega f_u(x,r_n(x))\dd x \quad \mbox{(where we used $\tilde r_n-r_n\to 0$ uniformly on $\Omega$)}\\
& = h(u,r).
\end{align*}
This concludes the proof of $\scm I\leq h$ on $\cC$ for elements $(u,r)\in \cC$ with $u\not =0$. So let us finally consider $(0,r)\in \cC$ which means that $r\in L^\infty(\Omega)$ is an arbitrary function. As before we can choose a sequence $r_n\stackrel{\ast}{\rightharpoonup}r$ in $L^\infty(\Omega)$ such that $\lim_{n\to\infty} \int_\Omega f_0(x,r_n(x))\dd x= h(0,r)$. Since $(0,r_n)\in \cC$ we get immediately that 
$$ 
\scm I(0,r)  \leq \liminf_{n\to\infty} \int_\Omega f_0(x,r_n(x))\dd x = h(0,r)
$$
and so the proof of $\scm I \leq h$ is finished. 
\end{proof}

We end the present section with a technical lemma which was used above and in Section \ref{sec_scmI=Istar}.

\begin{lemma} \label{lem:I_compare}
Recall the definition of $f_u$ in \eqref{def:dist:fix}. 
Let $u, u'\in H^1(\Omega)$, $r,r'\in \R$ and $x,x'\in \Omega$. Then 
    \begin{align*}
    |f_u(x,r)-f_{u'}(x',r')| & \leq 2\max\{\sqrt{ f_u(x,r)}, \sqrt{f_{u'}(x',r')}\} \interleave x-x', u(x)-u'(x'), r-r'\interleave \\
    & +  \interleave x-x', u(x)-u'(x'), r-r'\interleave^2
    \end{align*}
    where $\interleave\cdot\interleave$ is the Euclidean norm in $\R^{N+2}$.
\end{lemma}

\begin{proof} For $x, \tilde x\in \Omega$, $u, \tilde u\in (0,\infty)$ and $r,\tilde r\in \R$ let
\begin{align*}
\delta(x,\tilde x, u, \tilde u, r, \tilde r) \coloneqq |x-\tilde x|^2+ |u-\tilde u|^2+ |r-\tilde r|^2.
\end{align*}
Then we have 
\begin{eqnarray*}
\lefteqn{f_u(x, r)} \\
& = & \inf_{(\tilde x,\tilde u,\tilde r)\in \D} \delta(x,\tilde x, u(x), \tilde u,r, \tilde r) \\
& \leq &\inf_{(\tilde x,\tilde u,\tilde r)\in \D} \Bigl(\delta(x',\tilde x,u'(x'), \tilde u,r', \tilde r) + 2\sqrt{\delta(x',\tilde x,u'(x'), \tilde u, r', \tilde r)}\interleave x-x', u(x)-u'(x'), r-r'\interleave \\
&  & \quad + \interleave x-x', u(x)-u'(x'), r-r'\interleave^2\Bigr) \\
& = & f_{u'}(x', r') + 2 \sqrt{ f_{u'}(x', r')} \interleave x-x', u(x)-u'(x'), r-r'\interleave +\interleave x-x', u(x)-u'(x'), r-r'\interleave^2.
\end{eqnarray*}
If we switch the roles of $f_u(x,r)$ and $f_{u'}(x',r')$ we get a similar inequality from which the claim follows. 
\end{proof}

\section*{Acknowledgment}
L.B. and W.R. acknowledge funding by the Deutsche Forschungsgemeinschaft (DFG, German Research Foundation) – Project-ID 258734477 – SFB 1173.
L.B. and P.M. are members of the {\em Gruppo Nazionale per l'Analisi Ma\-te\-ma\-ti\-ca, la Probabilit\`a e le loro Applicazioni} (GNAMPA) of the {\em Istituto Nazionale di Alta Matematica} (INdAM). L.B and P.M. are partially supported by INdAM-GNAMPA Project 2026 titled \textit{Structural degeneracy and criticality in (sub)elliptic PDEs} (E53C25002010001).

\section*{Data availability}
We do not analyze or generate any datasets, because our work proceeds within a theoretical and
mathematical approach.

\section*{Conflict of interest}
The authors have no conflicts of interest to declare.

\bibliography{bibliography}
\bibliographystyle{abbrv}
\end{document}